\documentclass[11pt,reqno,a4paper]{amsart}
\usepackage{tikz}
\usetikzlibrary{shapes,arrows}
\usetikzlibrary{arrows.meta}
\usepackage{forest}
\usepackage{tikz-qtree}

\usetikzlibrary{arrows,shapes,positioning,shadows,trees}

\usepackage[hidelinks]{hyperref}

\usepackage{etoolbox}
\usepackage{amsmath}
\usepackage{enumerate}
\usepackage{amssymb}
\usepackage{amscd}
\usepackage{amsthm}
\usepackage{amsfonts}
\usepackage{graphicx}
\usepackage[all,cmtip]{xy}
\usepackage{enumitem}

\patchcmd{\subsection}{-.5em}{.5em}{}{}
\patchcmd{\subsubsection}{-.5em}{.5em}{}{}

\usepackage{enumitem}

\usepackage[T1]{fontenc}

\normalfont

\makeatother
\usepackage{hyperref}

\numberwithin{equation}{section}

\newcommand{\SL}{\operatorname{SL}}
\newcommand{\gsl}{\operatorname{\mathfrak{s}{l}}}

\newcommand{\kd}{\kappa}
\newcommand{\td}{\xi}

\newcommand{\xd}{X}
\newcommand{\pid}{\pi}
\newcommand{\sd}{s}
\newcommand{\ud}{u}
\newcommand{\md}{\mu}
\newcommand{\aad}{a}
\newcommand{\abd}{A}

\newcommand{\cC}{\mathcal{C}}
\newcommand{\cD}{\mathcal{D}}

\newcommand{\cI}{\mathcal{I}}

\newcommand{\cK}{\mathcal{K}}

\newcommand{\cP}{\mathcal{P}}

\newcommand{\cS}{\mathcal{S}}
\newcommand{\cT}{\mathcal{T}}

\newcommand{\bR}{\mathbb{R}}
\newcommand{\bS}{\mathbb{S}}

\newcommand{\bZ}{\mathbb{Z}}

\newcommand{\ra}{\rightarrow}

\newcommand{\qen}{\enskip \textrm{and} \enskip}
\newcommand{\qand}{\quad \textrm{and} \quad}

\newcommand\subsetsim{\mathrel{%
\ooalign{\raise0.2ex\hbox{$\subset$}\cr\hidewidth\raise-0.8ex\hbox{\scalebox{0.9}{$\sim$}}\hidewidth\cr}}}
\newcommand{\eps}{\varepsilon}

\DeclareMathOperator{\supp}{supp}

\DeclareMathOperator{\cum}{cum}

\DeclareMathOperator{\sgn}{sgn}
\DeclareMathOperator{\Diag}{Diag}
\DeclareMathOperator{\Leb}{Leb}

\DeclareMathOperator{\Vol}{Vol}

\newcommand{\bd}{\underline{d}}
\newcommand{\bi}{\underline{i}}
\newcommand{\bu}{\underline{u}}
\newcommand{\ba}{\underline{a}}

\tikzstyle{decision} = [diamond, draw, fill=blue!20, 
    text width=4.5em, text badly centered, node distance=3cm, inner sep=0pt]
\tikzstyle{block} = [rectangle, draw, fill=blue!20, 
    text width=5em, text centered, rounded corners, minimum height=2em]
\tikzstyle{line} = [draw, -latex']
\tikzstyle{cloud} = [draw, ellipse,fill=red!20, node distance=3cm,
    minimum height=2em]

\usepackage{mdframed}

\definecolor{lichtgrijs}{gray}{0.95}
\mdfdefinestyle{mystyle}{ %
    backgroundcolor=lichtgrijs, %
    linewidth=0pt, %
    innertopmargin=10pt, %
    innerbottommargin=10pt,%
    nobreak=true
}

\newmdtheoremenv[style=mystyle]{theorem}{Theorem}[section]
\newmdtheoremenv[style=mystyle]{proposition}[theorem]{Proposition}
\newmdtheoremenv[style=mystyle]{lemma}[theorem]{Lemma}
\newmdtheoremenv[style=mystyle]{corollary}[theorem]{Corollary}

\theoremstyle{definition}
\newtheorem{definition}[theorem]{Definition}
\newtheorem{remark}[theorem]{Remark}

\tikzstyle{decision} = [diamond, draw, fill=blue!20, 
    text width=4.5em, text badly centered, node distance=3cm, inner sep=0pt]
\tikzstyle{block} = [rectangle, draw, fill=blue!20, 
    text width=5em, text centered, rounded corners, minimum height=2em]
\tikzstyle{line} = [draw, -latex']
\tikzstyle{cloud} = [draw, ellipse,fill=red!20, node distance=3cm,
    minimum height=2em]

\renewcommand\labelenumi{(\roman{enumi})}
\renewcommand\theenumi\labelenumi

\usepackage{changepage}

\begin{document}
\bibliographystyle{plain} 

\title{Counting in generic lattices and higher rank actions}

\author{Michael Bj\"orklund}

\address{Department of Mathematics, Chalmers, Gothenburg, Sweden}
\email{micbjo@chalmers.se}

\author{Alexander Gorodnik}
\address{Institute f\"ur Mathematik, University of Z\"urich, Switzerland}
\email{alexander.gorodnik@math.uzh.ch}

\subjclass[2010]{Primary: 11K60, 11N45. Secondary: 37A25, 60F05}
\keywords{Counting problems, central limit theorems, exponential mixing of all orders}

\date{}

\begin{abstract}
We consider the problem of counting lattice points contained in 
domains in $\bR^d$ defined by products of linear forms and we 
show that the normalized discrepancies in these counting problems
satisfy non-degenerate Central Limit Theorems, provided that $d \geq 9$. We also  study more refined versions pertaining to "spiraling of approximations". Our techniques are dynamical in nature and exploit effective exponential mixing of all orders for actions of higher-rank abelian groups
on the space of unimodular lattices.
\end{abstract}

\maketitle


\section{Introduction}

Let $\Lambda$ be a lattice in $\bR^d$, and let $(\Omega_T)$ be an increasing family of Borel subsets of $\bR^d$ with finite volumes tending to infinity as $T\to\infty$. 
A fundamental problem in the Geometry of Numbers is to estimate 
the number of points in  $\Lambda$ which are contained in $\Omega_T$.
Under mild regularity conditions, one can usually show that 
$$
|\Lambda\cap  \Omega_T| = \frac{\Vol(\Omega_T)}{\Vol(\bR^d/\Lambda)}+o\Big(\Vol(\Omega_T)\Big)\quad\quad\hbox{as $T\to\infty$.}
$$ 
In this paper we study the corresponding \emph{discrepancy function} defined by
\begin{equation}
\label{def_RT}
\cD_{T}(\Lambda) := |\Lambda\cap \Omega_T | - \frac{\Vol(\Omega_T)}{\Vol(\bR^d/\Lambda)}.
\end{equation}
When the domain $\Omega_T$ is a $T$-dilation a region $\Omega \subset \bR^d$ with piecewise smooth boundary, one can easily prove that 
\begin{equation}
\label{eq:gen_n}
\cD_{T}(\Lambda)=O_{\Lambda}\Big(\Vol(\Omega_T)^{1-1/d}\Big),
\end{equation}
and this estimate is the \emph{best possible} in this generality. However, the estimate has been improved for certain particular classes of domains. A well-studied setting is when the domain $\Omega$ has non-vanishing curvature. In this case, Hlawka \cite{hl} has shown that 
\begin{equation}
\label{eq:hl1}
\cD_{T}(\Lambda)=O_{\Lambda}\Big(\Vol(\Omega_T)^{1-2/(d+1)}\Big)
\end{equation}
and
\begin{equation}
\label{eq:hl2}
\cD_{T}(\Lambda)=\Omega_{\Lambda}\Big(\Vol(\Omega_T)^{1-(d+1)/(2d)}\Big).
\end{equation}
These bounds have been subsequently improved by a number of people 
(see, for instance, \cite{i} for a survey).

\medskip

In this paper we shall be interested in asymptotic behaviour ($T \ra \infty$) of the discrepancy function $\cD_{T}(\Lambda)$ for "generic" lattices $\Lambda$. The following two questions naturally arise in this setting:
\begin{enumerate}
	\item[(i)] what is the asymptotic "generic" growth of $\cD_{T}(\Lambda)$? \vspace{0.1cm}
	\item[(ii)] do suitably normalized discrepancy functions converge in distribution?
\end{enumerate}
Concerning Question (i): it turns out that the estimate \eqref{eq:gen_n} can be improved for generic lattices.
The first striking result in this direction was established by W.~Schmidt~\cite{Sch60(2)}. He proved that for a \emph{every} increasing family of 
Borel sets $\Omega_T$ as above and almost every lattice $\Lambda$,
$$
\cD_{T}(\Lambda)=O_{\Lambda,\eps}\Big(\Vol(\Omega_T)^{1/2+\eps}\Big)\quad \hbox{for all $\eps>0$.}
$$
However, the exact asymptotic behavior of $\cD_{T}(\Lambda)$ for generic lattices
is still quite mysterious, and it turns out that the answer depends very sensitively on the shape of the domains. 
For instance, Hardy, Littlewood \cite{hl} and Khinchin \cite{kh} 
discovered that when $\Omega_T$ is a $T$-dilation 
of a generic compact polygon in $\bR^2$, then 
$$
\cD_{T}(\bZ^2)=O_{\eps}\Big(\big(\log \Vol(\Omega_T)\big)^{1+\eps}\Big)\quad \hbox{for all $\eps>0$.}
$$
This exhibits a striking difference with the estimates \eqref{eq:hl1}--\eqref{eq:hl2}
for strictly convex domains. Skriganov \cite{skr} established a far-reaching generalization of this estimate. He showed that when $\Omega_T$ is 
a dilation by a factor $T$ of a compact polyhedron in $\bR^d$, then for almost every unimodular lattice $\Lambda$,
$$
\cD_{T}(\Lambda)=O_{\Lambda,\eps}\Big(\big(\log \Vol(\Omega_T)\big)^{d-1+\eps}\Big)\quad \hbox{for all $\eps>0$.}
$$
It is not known whether the above bound is optimal.
Another well-studied example is the case when 
the domains $\Omega_T$ are the Euclidean balls in $\bR^d$.
In this case, it was shown by Kelmer \cite{kel} that 
for any exponentially growing sequence $T_i\to\infty$ and almost all lattices $\Lambda$,
$$
\cD_{T_i}(\Lambda)=O_{\Lambda,\eps}\Big(\Vol(\Omega_{T_i})^{1-(d+1)/(2d) +\eps}\Big)\quad \hbox{for all $\eps>0$.}
$$

\medskip

Concerning Question(ii) above: 
several results have been proved for certain particular families of lattices.
For instance, it was discovered by Beck that the distributions of suitably normalized
discrepancy functions are asymptotically Gaussian.
We refer to a survey \cite{beck1} and a monograph \cite{beck2} for 
a comprehensive exposition of these results. Beck
considered the domains
$$
\Omega_T:=\big\{(x,y)\in \bR^2\,:\,\, x^2-2y^2\in (a,b),\, 0<x< T,\, y>0  \big\}  
$$
and translated lattices $\Lambda_{\omega}:=\bZ^2+(\omega,0)$ with $0<\omega <1$
and showed that there exists an explicit $\sigma>0$ such that
\begin{equation}\label{eq:bek}
\hbox{Leb}\big(\big\{ \omega\in (0,1): \Vol(\Omega_T)^{-1/2} \mathcal{D}_T(\Lambda_\omega)<\xi\big\}\big)
\longrightarrow \frac{1}{\sigma \sqrt{2\pi}} \int_{-\infty}^\xi e^{-t^2/2\sigma^2} \, dt
\quad \textrm{as $T\to\infty$}.
\end{equation}
While this approach seems to work for domains defined by more general indefinite integral binary quadratic forms, it was not clear whether this result could hold in higher dimensions
since its proof was based on properties of continued fraction expansions for quadratic irrational. Furthermore, Beck points out that there are essential difficulties
in extending his work to higher dimensions related to the long-standing Littlewood Conjecture. 

Levin \cite{lev} investigated the discrepancy 
function of the family of lattices of the form
$$
\Lambda_{\ba}:=\hbox{diag}(a_1,\ldots,a_d)^{-1}\mathcal{O},\quad \ba=(a_1,\ldots,a_d)\in (0,1)^d,
$$
where $\mathcal{O}$ is a fixed lattice in $\bR^d$ arising from an order in a totally real number field. He showed that for the boxes $\Omega_{\underline{N}}:=[-N_1,N_1]\times \cdots \times  [-N_d,N_d]$, 
then suitably normalized  discrepancy
functions $\mathcal{D}_{\underline{N}}(\Lambda_{\ba})$ are asymptotically Gaussian as $N_1\cdots N_d\to \infty$, with $\ba\in (0,1)^d$ considered random.
Since the results \cite{beck1,beck2,lev} 
treat only very particular lattices arising from 
orders in number fields, one may wonder whether this behavior occurs 
for truly generic lattices. We will address this question in the present paper.

One should also mention the ground-breaking works of Dolgopyat, Fayad \cite{df1,df2}
(see also the survey \cite{df0}), generalizing Kersten \cite{k1},
about the discrepancy of distribution for toral translations.
Using our terminology, these results can be interpreted 
in terms of discrepancy functions for the family of lattices given by
$$
\Lambda_{\underline{u}}:=\big\{(x_1+ u_1y,\ldots, x_{d-1}+ u_{d-1}y, y):\, (x_1,\ldots,x_{d-1},y)\in \bZ^d\big\}\quad\hbox{with}\;\;  0\le u_1,\ldots,u_{d-1}<1.
$$
and certain families of domains $\Omega_{T}(\underline{\theta})$ depending on additional parameters
$\underline{\theta}$. It is shown in \cite{df1,df2} that the corresponding discrepancy for $|\Lambda_{\underline{u}}\cap \Omega_{T}(\underline{\theta})|$ after a suitable normalization 
converges in distribution as $T\to \infty$, with $(\underline{u},\underline{\theta})$ considered random.
It should be noted that the obtained limit distributions in \cite{df1,df2} are different from the Normal Law.
Further related results about distribution of 
Diophantine approximants were proved in \cite{DFV} and \cite{BG2}.

\subsection{Main results}

Let $L_1,\ldots, L_d:\bR^d\to \bR$ be linearly independent 
linear forms and $N_L(x):= L_1(x)\cdots L_d(x)$. 
For a bounded interval $I\subset \bR^+$ and $T>0$, we consider the domains
$$
\Omega_T(I):=\big\{x\in\bR^d:\, N_L(x)\in I\;\;\hbox{and}\;\; 0<L_1(x),\ldots,L_d(x)< T  \big\}.
$$
We write $X$ for the space 
of unimodular lattices in $\bR^d$ equipped with the unique $\SL_d(\bR)$-invariant probability measure $\mu$.
The following result provides an analogue of \eqref{eq:bek}
for $\mu$-generic unimodular lattices:
\vspace{0.2cm}
 
\begin{theorem}\label{main1}
Let $\cD_T$ denote the discrepancy function for $\Omega_T(I)$. If $d \geq 9$, then 
$$
\mu\Big(\big\{ \Lambda\in X:\, \Vol(\Omega_T)^{-1/2} \mathcal{D}_T(\Lambda)<\xi\big\}\Big)
\longrightarrow \frac{1}{\sigma(I) \sqrt{2\pi}} \int_{-\infty}^\xi e^{-t^2/2\sigma(I)^2} \, dt
\quad \textrm{as $T\to\infty$},
$$
for all $\xi \in \bR$, where
\[
\sigma(I)^2 := \frac{1}{\zeta(d)} \sum_{p,q=1}^\infty  \frac{\Leb\left(p^d I \cap q^d I\right)}{p^d q^d \Leb(I)}.
\]
\end{theorem}

\vspace{0.3cm}

Athreya, Ghosh and Tseng \cite{ATG} studied the related 
problem of "spiraling" of Diophantine approximants
which involves  counting the lattice points 
in the domains
$$
 \Big\{ (x,y) \in \bR^{d-1} \times \bR :\,\, \|x\|\cdot  |y| \in I, \enskip \frac{x}{\|x\|} \in B, \enskip 0<\|x\|< T, \enskip 0 < y < T \Big\},
$$
defined for an interval $I\subset \bR^+$ and a Borel subset $B\subset S^{d-1}$. 
Our method allows to analyze the distribution of the error term for this counting problem. \\

More generally, for $k \geq 2$ and  positive integers $d_1,\ldots,d_k$,
we set 
\[
\bd = (d_1,\ldots,d_k) \qand d = d_1 + \ldots + d_k,
\]
and define
$\bS_{\bd} := \prod_{j=1}^k S^{d_j-1},$
where $S^{d_j-1}$ denotes the unit sphere in $\bR^{d_j}$, endowed with the standard Euclidean inner product, with the 
convention that $S^0 = \{-1,1\}$. The corresponding norm on $\bR^{d_j}$ will be denoted by $\|\cdot\|$, the 
spherical measure on $S^{d_j-1}$ will be denoted by $\kappa_{j}$, and we set 
\begin{equation}\label{eq:k}
\kd := \kappa_{1} \otimes \cdots \otimes \kappa_{k}.
\end{equation}
Let us also fix rotation-invariant smooth metrics on each $S^{d_j-1}$ with $d_j \geq 2$. If $d_j = 1$, we endow $S^0 = \{-1,1\}$ with the
discrete distance. If $B \subset \bS_{\bd}$ is a Borel set and $\eps > 0$, we denote by $B_\eps$ the $\eps$-thickening of $B$ with 
respect to the products of the chosen metrics. 
We say that a Borel set $B \subset \bS_{\bd}$ has a \emph{smooth boundary} if 
\[
\kd(B_\eps) - \kd(B) \ll \eps,  \quad \textrm{for all small enough $\eps > 0$},
\]
where the implicit constants are independent of $\eps$. \\

Let now $L_j:\bR^d\to \bR^{d_j}$, $j=1,\ldots,k$, be linear maps such that $(L_1,\ldots,L_k)$ is a bijection of $\bR^d$.
We define 
\begin{equation}
\label{defNdThetad}
N_L(z): = \prod_{j=1}^k \|L_j(z)\|^{d_j} \qand \td_L(z): = \Big( \frac{L_1(z)}{\|L_1(z)\|}, \ldots, \frac{L_k(z)}{\|L_k(z)\|} \Big).
\end{equation}
Given a bounded interval $I \subset (0,\infty)$, a Borel set $B \subset \bS_{\bd}$ and $T>0$, we consider the domains
\begin{equation}
\label{defOT_0}
 \Omega_T(I,B) := \big\{ z \in \bR^{d} \, :\, \, N_L(z) \in I, \enskip \td_L(z) \in B \qen 0<\|L_1(z)\|,\ldots, \|L_1(z)\| <T \big\}.
\end{equation}
Our main result is the following:
\vspace{0.3cm}

\begin{theorem}\label{main}
	When $k\ge 2$ and $d\ge 9$, the discrepancy functions for the sets $\Omega_T(I,B)$ satisfy,	
	$$
	\mu\Big(\big\{ \Lambda\in X:\, \Vol(\Omega_T)^{-1/2} \mathcal{D}_T(\Lambda)<\xi\big\}\Big)
	\longrightarrow \frac{1}{\sigma(I,B) \sqrt{2\pi}} \int_{-\infty}^\xi e^{-t^2/2\sigma(I,B)^2} \, dt
	\quad \textrm{as $T\to\infty$},
	$$
	 for all $\xi\in\bR$, where
\[
\sigma(I,B)^2 := 
\frac{1}{\zeta(d)} \left(\sum_{p,q=1}^\infty  \frac{\Leb\left(p^d I \cap q^d I\right)}{p^d q^d \Leb(I)}\right) \Big( 1 +   \frac{\kd(B \cap - B)}{\kd(B)} \Big).
\]
\end{theorem}

\vspace{0.2cm}

Theorems \ref{main1} and \ref{main} have been announced in \cite{BG0} for $d\ge 4$.
However, it turned out that the technical part of our argument works only for $d\ge 9$. 

\medskip

In the next section,
we summarize the main steps of the proof of Theorem \ref{main}.
Our argument can be roughly divided into two parts that involve:
\begin{itemize}
	\item a construction 
	of a suitable approximation for the counting function (Section \ref{sec:smoothapprox}), \vspace{0.1cm}
	\item analysis of such approximations (Section \ref{sec:general}). 
\end{itemize}  

\subsection*{Acknowledgements}

The authors are grateful for the hospitality shown by the departments of Mathematics at Bristol University, Chalmers University
and University of Z\"urich, where most of this work was done. The first author also acknowledges financial support from 
GoCas Young Excellence grant 11423310 until December 2018, and from the grant VR M Björklund 20-23 from the Swedish Research Council from January 2020.
The second author was supported by SNF grant 200021--182089.

\section{Outline of the proof}\label{sec:outline}

Our argument will involve analysis on the space $X$ of unimodular lattices in $\bR^d$, which can be considered as a homogeneous space
$X\simeq \hbox{SL}_d(\bR)/\hbox{SL}_d(\bZ)$.
The space $X$ supports a unique $\hbox{SL}_d(\bR)$-invariant probability measure, which we shall denote by $\mu$ throughout the paper. \\ 
Given  a bounded Borel measurable function 
$f : \bR^d \ra \bR$  with bounded support, its \emph{Siegel transform}
	$\widehat{f} : \xd \ra \bR$  is defined by
	\[
	\widehat{f}(\Lambda) := \sum_{z \in \Lambda \setminus \{0\}} f(z), \quad \textrm{for $\Lambda \in \xd$}.
	\]
According to Siegel's Mean Value Theorem \cite{Sie},
if $f$ is Riemann integrable, then
\begin{equation}
\label{eq:siegel}
\int_{\xd} \widehat{f} \, d\md = \int_{\bR^d} f(z) \, dz,
\end{equation}
where we normalise the Lebesgue measure $dz$ on $\bR^d$ so that the unit cube is assigned volume one. \\

Suppose that $\Omega_T$ is a bounded Borel set 
in $\bR^d$, which do not contain the origin. Then, with the above notations, 
\begin{equation}
\label{eq:siegel2}
|\Omega_T \cap \Lambda| = \widehat{\chi}_{\Omega_T}(\Lambda) \qand \Vol(\Omega_T) = \int_{\xd} \widehat{\chi}_{\Omega_T} \, d\md,
\end{equation}
so that 
$$
\mathcal{D}_T(\Lambda)=\widehat{\chi}_{\Omega_T}(\Lambda)-\int_{\xd} \widehat{\chi}_{\Omega_T} \, d\md.
$$
In the setting of Theorem \ref{main}, these formulas can rewritten further. 
In what follows, we retain the notation used 
there. In particular, we have fixed $k \geq 2$ and $d \geq 3$, as well as 
a $k$-tuple $\bd = (d_1,\ldots,d_k)$ of positive integers with $d = d_1 + \ldots + d_k$. We have chosen a bounded interval 
$I \subset (0,\infty)$ and a Borel set $B \subset \bS_{\bd}$ with a smooth boundary. We 
denote by $\Omega_T = \Omega_T(I,B)$ the sets defined in \eqref{defOT_0}.
For simplicity, we suppose that the maps $L_j$ are the standard coordinate projections.
Then the domains $\Omega_T$ can be conveniently foliated
by the level sets 
\[
\mathcal{L}_{s,\xi}:=\big\{z \in \bR^d \, : \,\, N_L(z) = s \qen \td_L(z) = \xi\big\}, \quad \textrm{for $s \in I$ and $\xi \in B$},
\] 
which are invariant under the subgroup $\abd < \SL_d(\bR)$ of diagonal matrices of the form
\begin{equation}
\label{defb}
\aad(u) := \Diag\Big(e^{u_1}I_{d_1},e^{u_2}I_{d_2},\ldots,e^{u_{k-1}}I_{d_{k-1}},e^{-\frac{1}{d_k}\sum_{j=1}^{k-1} d_j u_j}I_{d_k} \Big), \quad 
\textrm{for $u \in \bR^{k-1}$}.
\end{equation}
We note that $\abd\simeq \bR^{k-1}$ since 
$$
\aad(u)\aad(v) = \aad(u+v)\quad\hbox{ for 
	all $u,v \in \bR^{k-1}$.}
$$
The initial idea of our approach is that the level sets 
$\mathcal{L}_{s,\xi}$
can be tessellated, using the action of a discrete subgroup of $\abd$ on $\bR^d$. Unfortunately, the domains $\Omega_T$ themselves do not possess such simple tillings. However, it turns out that each of the intersections $\Omega_T\cap \mathcal{L}_{s,\xi}$ has a tilling where tiles and 
the discrete subgroup depends on the parameters $s$ and $T$ (but not on the parameter $\xi$). 
We will show that the indicator functions $\chi_{\Omega_T}$ can be approximated by suitable integrals of 
varying functional averages. These ``functional tilings'' stem from the above tilings for different values of $s$ and $\xi$ and are constructed 
using the following data:
\begin{itemize}
	\item a collection of finite measure spaces $(Y_{T,i},\kappa_{T,i})$ indexed by $T>0$ and $i$ in a finite set $\mathcal{I}$, \vspace{0.1cm}
	\item a collection of bounded Borel functions $f_{T,i}:\bR^d\times Y_{T,i}\to [0,\infty)$ with $T>0$ and $i\in\mathcal{I}$, \vspace{0.1cm}
	\item a collection of finite subsets $Q(y_i)$ of $A$ with $y_i\in Y_{T,i}$.
\end{itemize}
The corresponding ``functional tilling'' is given by
\begin{equation}\label{FT0}
F_T(z) := \sum_{i\in\mathcal{I}} \int_{Y_{T,i}} \Big( \sum_{a \in Q_{T,i}(y_i)} f_{T,i}(az,y_i) \Big) \, d\kappa_{T,i}(y_i), \quad \textrm{for $z \in \bR^d$}.
\end{equation}
We shall show that for a suitable choice of the data, $F_T$ provides 
an approximation for the characteristic function $\chi_{\Omega_T}$
in the  sense that
\begin{equation*}
\big\|\chi_{\Omega_T} - F_T\big\|_1 = o\big(\Vol(\Omega_T)^{1/2}\big)
\quad \hbox{as $T\to\infty$.}
\end{equation*}
Assuming this, we can then write
\begin{equation*}
\frac{\widehat{\chi}_{\Omega_T} - \hbox{vol}(\Omega_T)}{\Vol(\Omega_T)^{1/2}} 
= 
\frac{\widehat{\chi}_{\Omega_T} - \widehat{F}_T}{\Vol(\Omega_T)^{1/2}} 
+
\frac{\widehat{F}_T - \int_{\xd} \widehat{F}_T \, d\md}{\Vol(\Omega_T)^{1/2}}
+
\frac{\int_{\xd} \big( \widehat{F}_T - \widehat{\chi}_{\Omega_T} \big) \, d\md}{\Vol(\Omega_T)^{1/2}},
\end{equation*}
where the first and third term on the right hand side tend to zero in the $L^1(\md)$-norm. Thus, the distributional limit of $\cD_T(\Lambda)$ is the same as 
the distributional limit of the sequence of functions
\[
\Upsilon_T(\Lambda):= \hbox{Vol}(\Omega_T)^{-1/2}\left(\widehat{F}_T(\Lambda) - \int_{\xd} \widehat{F}_T \, d\md\right).
\]
The significance of this observation is that 
averages $F_T$  can be investigated using homogeneous dynamics techniques. \\

Since the averages of the form also arise in other arithmetic problems, 
we will analyze their behavior in an abstract axiomatic setting (cf. assumptions (I.a)--(I.c) and (II.a)--(II.c) below).
This analysis will be carried out in Section \ref{sec:general}.
Our main result here is Theorem \ref{thm_criterion}. Notably, it shows that 
when certain basic norm estimates for functions $f_{T,i}$ hold,
the distributional convergence of $\Upsilon_T(\Lambda)$ hold
provided that only the variance $\|\Upsilon_T\|_{L^2(X)}$ converges.
Next, in Section \ref{sec:smoothapprox} we construct an approximation 
for $\chi_{\Omega_T}$ of the form \eqref{FT0} satisfying our assumptions 
(I.a)--(I.c) and (II.a)--(II.c). Once such an approximation
is available, our main result will be a corollary of 
Theorem \ref{thm_criterion}.

\section{Analysis of general functional tillings} \label{sec:general}

In this section we consider a family of functions $F_T$ on $\bR^d$ 
defined by a ``functional tilling'' as in \eqref{FT0}.
Our goal is to analyze the asymptotic behavior of the sums $\widehat{F}_T(\Lambda)=\sum_{z\in \Lambda\backslash \{0\}} F_T(z)$ for lattices $\Lambda$ in $\bR^d$. 
We will pose several assumptions on the objects defining $F_T$ and then in the next section demonstrate that the developed framework applies to our setting.
We have opted for this axiomatic approach because it could be useful for other counting problems, and it makes easier to follow the details of quite technical approximations arguments. Our main result here is Theorem \ref{thm_criterion},
which establishes the Central Limit Theorem for 
$(\widehat{F}_T)$, with respect to the measure $\mu$.

\subsection{Functional averages and their truncations}
\label{subsec:data}

Let $\mathcal{I}$ be a finite set. For $T>0$ and $i\in\mathcal{I}$, we consider:
\begin{enumerate}
\item[(I.a)] finite measure spaces $(Y_{T,i},\kappa_{T,i})$ satisfying $\sup_{T,i} \kappa_{T,i}(Y_{T,i}) < \infty$, \vspace{0.2cm}
\item[(I.b)] bounded Borel functions $f_{T,i} :\bR^d \times Y_{T,i} \ra [0,\infty)$ such that for $y_i \in Y_{T,i}$, the map $x \mapsto f_{T,i}(x,y_i)$
is smooth, and supported in a compact set $\cK \subset \bR^d$, independent of $T$, $i$, and $y_i$, \vspace{0.2cm}
\item[(I.c)] a set-valued map $y_i \mapsto Q_{T,i}(y_i)$ from $Y_{T,i}$ into the set of finite subsets of $\abd$ such that
$\sup_{i,y_i} |Q_{T,i}(y_i)|\ll V_T$ with a parameter $V_T$ satisfying $V_T\to \infty$ as $T\to\infty$.
\end{enumerate}
\vspace{0.2cm}

For $f \in C^\infty_c(\bR^d)$, let $\partial_k f$ denote the partial derivative of $f$ with 
respect to the $k$-th coordinate for $k=1,\ldots,d$. If $\beta = (\beta_1,\ldots,\beta_d)$ is a multi-index, we set
$\partial_\beta f = \partial_1^{\beta_1} \cdots \partial_d^{\beta_d}f$, and define
\begin{equation}
\label{defCp}
\|f\|_{C^p} = \max_{|\beta| \leq p} \|\partial_\beta f\|_{\infty}, \quad \textrm{for $p \geq 1$},
\end{equation}
where $|\beta| = \beta_1 + \ldots + \beta_d$. \\

\noindent We use the notations
\begin{align}\label{defM0_0}
M_{T} &:= \max_{i\in\mathcal{I}} \int_{Y_{T,i}} \big\|f_{T,i}(\cdot,y_i)\big\|_{\infty} \, d\kappa_{T,i}(y_i), \\
\label{defM0}
M_{T,q} &:= \max_{i\in\mathcal{I}} \sup_{y_i \in Y_{T,i}} \big\|f_{T,i}(\cdot,y_i)\big\|_{C^q}.
\end{align}
Given the data in (I.a)--(I.c), we consider the family of functions given by
\begin{equation}
\label{FT0_1}
F_T(z) := \sum_{i\in\mathcal{I}} \int_{Y_{T,i}} \Big( \sum_{a \in Q_{T,i}(y_i)} f_{T,i}(az,y_i) \Big) \, d\kappa_{T,i}(y_i), \quad\hbox{ for $z\in\mathbb{R}^d$,}
\end{equation}
and their Siegel transforms
\begin{equation}
\label{FT0_2}
\widehat{F}_T(\Lambda) = \sum_{i\in\mathcal{I}} \int_{Y_{T,i}} \Big( \sum_{a \in Q_{T,i}(y_i)} \widehat{f}_{T,i}(a\Lambda,y_i) \Big) \, d\kappa_{T,i}(y_i), \quad\hbox{ for $\Lambda\in X$.}
\end{equation}
Our goal is to show that under suitable assumptions the functions
\[
\Upsilon_T(\Lambda):= V_T^{-1/2}\left(\widehat{F}_T(\Lambda) - \int_{\xd} \widehat{F}_T \, d\md\right)
\]
converge in distribution. One of the difficulties that arises here is that
Siegel transforms (even for  bounded Borel functions with bounded support) are \emph{not} bounded. Nonetheless, they are typically only large
on sets of very small $\md$-measure and belong to $L^p(X)$ for $p<d$. 
Here and latter in the paper we always assume that $d\ge 3$ so that the Siegel transforms are $L^2$-integrable.
This makes it possible to efficiently approximate a Siegel transform by bounded functions on $\xd$ whose $L^p$-distance from the original Siegel transform is small.   
To make this approximation precise, we shall use a family of compactly supported cutoff function $\eta_L : \xd \ra [0,1]$ with $L>0$, constructed in \cite[Lemma~4.10]{BG2} 
such that for every compact
set $K \subset \bR^d$ and $f \in C(K)$, we have 
\begin{equation}\label{eq:eta1}
\left\|\widehat{f} \, \eta_L\right\|_{L^\infty(X)} \ll_K L \|f\|_{\infty}.
\end{equation}
Furthermore, for every $\eps>0$, 
\begin{equation}\label{eq:eta2}
\left\|\widehat{f} \, (1-\eta_L)\right\|_{L^1(X)} \ll_{K,\eps} L^{-d+1+\eps} \|f\|_{\infty} \qand \left\|\widehat{f} \, (1-\eta_L)\right\|_{L^2(X)} \ll_{K,\eps} L^{-d/2+1+\eps} \|f\|_{\infty},
\end{equation}
where the implicit constants are independent of $L$. \\

We introduce a parameter $L_T\to \infty$, which will be specified later, 
and introduce the functions
$\varphi_{T,i} : \xd \times Y_{T,i} \ra [0,\infty)$ defined by
\[
\varphi_{T,i}(\Lambda,y_i) := \widehat{f}_{T,i}(\Lambda,y_i) \eta_{L_T}(\Lambda), \quad \textrm{for $\Lambda \in \xd$ and $y_i \in Y_{T,i}$},
\]
which provide compactly supported truncations of the functions $\widehat{f}_{T,i}(\cdot,y_i)$.
We then consider
\[
\Phi_{T}(\Lambda) := \sum_{i\in \cI} \int_{Y_{T,i}} \Big( \sum_{a \in Q_{T,i}(y_i)} \varphi_{T,i}(a\Lambda,y_i) \Big) \, d\kappa_{T,i}(y_i), \quad \textrm{for $\Lambda \in \xd$}.
\]
The following lemma shows that this function approximates the Siegel transform $\widehat{F}_T$ if the parameter $L_T$ grows fast enough.

\begin{lemma}
\label{lemma_trunc}
If for some $\eps>0$,
\begin{equation}\label{eq:cond1}
L_T^{-d/2+1+\eps}\, V_T^{1/2} \, M_{T}\to 0\quad \textrm{as $T \ra \infty$},
\end{equation}
then $$
\left\|\widehat{F}_T - \Phi_{T}\right\|_{L^2(X)} = o\left(V_T^{1/2}\right)
\quad \textrm{as $T \ra \infty$}.
$$
Similarly, if 
\begin{equation}\label{eq:cond1_1}
L_T^{-d+1+\eps}\, V_T^{1/2} \, M_{T} \to 0 \quad \textrm{as $T \ra \infty$},
\end{equation}
then $$
\left\|\widehat{F}_T - \Phi_{T}\right\|_{L^1(X)} = o\left(V_T^{1/2}\right)
\quad \textrm{as $T \ra \infty$}.
$$
\end{lemma} 

\vspace{0.2cm}

Before we proceed to the proof of this lemma, we discuss its relevance to our arguments so far. We wish to prove convergence in distribution for
the functions  
\[
 \Upsilon_T= V_T^{-1/2}\left(\widehat{F}_T - \Phi_{T}\right) + V_T^{-1/2}\left(\Phi_{T} - \int_{\xd} \Phi_{T} \, d\md\right) + V_T^{-1/2}\int_X\Big( \Phi_{T} - \widehat{F}_T\Big)\, d\mu.
\]
If $L_T$ is chosen as in \eqref{eq:cond1_1}, then the first and third term of  the right hand side tend to zero in the $L^1$-norm, whence $\Upsilon_T$ converges in distribution to a continuous measure if 
and only if the  functions 
\begin{equation}\label{eq:psi}
\Psi_T := V_T^{-1/2}\left(\Phi_{T} - \int_{\xd} \Phi_{T} \, d\md\right) 
\end{equation}
do.
In the upcoming subsections, we will analyse this type of sequences.

\begin{proof}[Proof of Lemma \ref{lemma_trunc}]
By construction, we have
\[
\left\|\widehat{F}_T - \Phi_{T}\right\|_{L^2(X)} \leq \sum_{i\in\mathcal{I}} \int_{Y_{T,i}} \sum_{a \in Q_{T,i}(y_i)} \left\|\big(\widehat{f}_{T,i}(\cdot,y_i)\circ a\big) (1-\eta_{L_T}\circ a)\right\|_{L^2(X)} \, d\kappa_{T,i}(y_i).
\]
Since the measure $\md$ is $\abd$-invariant, the inner terms are independent of $a \in Q_{T,i}(y_i)$, whence
\[
\left\|\widehat{F}_T - \Phi_{T}\right\|_{L^2(X)} \leq \sum_{i\in\mathcal{I}} \int_{Y_{T,i}} |Q_{T,i}(y_i)| \,  \left\|\widehat{f}_{T,i}(\cdot,y_i)(1-\eta_{L_T})\right\|_{L^2(X)} \, d\kappa_{T,i}(y_i).
\]
By the assumption (I.b), the supports of the functions $x \mapsto f_{T,i}(x,y_i)$ are all contained in a fixed compact set $\cK \subset \bR^d$, independent of  $T, i$ and $y_i$. Hence, by \eqref{eq:eta2},
\[
\left\|\widehat{f}_{T,i}(\cdot,y_i)(1-\eta_{L_T}(\cdot))\right\|_{L^2(X)} \ll_{\cK,\eps} L_T^{-d/2+1+\eps} \|f_{T,i}(\cdot,y_i)\|_{\infty}, \quad \textrm{for all $y_i \in Y_{T,i}$}.
\]
Furthermore, by the assumption (I.c), we have $|Q_{T,i}(y_i)| \le V_T$, so that we conclude that
\[
\left\|\widehat{F}_T - \Phi_{T}\right\|_{L^2(X)} \ll_{\cK,\eps}  L_T^{-d/2+1+\eps}
V_T \, \left(\sum_{i\in\mathcal{I}}\int_{Y_{T,i}}\|f_{T,i}(\cdot,y_i )\|_{\infty} \, d\kappa_{T,i}(y_{T,i}) \right).
\]
This implies the first part of the lemma, and the proof of the second part is similar.
\end{proof}

\subsection{Sobolev norms and mixing estimates}
\label{subsec:sobolev}
In order to obtain quantitative estimates on correlations, we need to control the smoothness of the functions. Our main
tool for this purpose are Sobolev norms, which we now introduce. First note that every $Y$ in the Lie algebra $\gsl_d(\bR)$ of $\SL_d(\bR)$ induces a differential operator $D_Y$ on $C^\infty(\xd)$ by
\[
(D_Y\varphi)(\Lambda) = \frac{d}{dt}\varphi(e^{tY}\Lambda) \mid_{t = 0}
\quad \hbox{for functions $\phi$ on $X$.}
\]
More generally, if we fix a basis $Y_1,\ldots,Y_m$ of $\gsl_d(\bR)$ with $m = d^2-1$, and if $Y$ is a monomial 
in the universal enveloping algebra of $\gsl_d(\bR)$ with respect to this basis, say $Y = Y_1^{\eta_1} \cdots Y_m^{\eta_m}$ for 
non-negative integers $\eta_1,\ldots,\eta_m$, then we define $D_\eta := D_{Y_1}^{\eta_1} \cdots D_{Y_m}^{\eta_m}$, and 
refer to the integer $|\eta| := \eta_1 + \ldots + \eta_m$ as the \emph{order} of $D_\eta$, where $\eta = (\eta_1,\ldots,\eta_m)$. We recall that
the Birkhoff-Poincar\'{e}-Witt theorem guarantees that the set of all (ordered) monomials with respect to our chosen basis $\{Y_1,\ldots,Y_m\}$
form a basis for the universal enveloping algebra of $\gsl_d(\bR)$. 
We write $C_c^\infty(\xd)$ for the space of compactly supported functions $\phi$ such that all the derivatives $D_\eta\phi$ exist. \\

Let $\Lambda \in \xd$. We say that a linear subspace $V < \bR^d$ is \emph{$\Lambda$-rational} if 
$V \cap \Lambda$ is a lattice in $V$. If $V$ is $\Lambda$-rational, we denote by $d_\Lambda(V)$ the volume of $V/V \cap \Lambda$.
We define
\[
\alpha(\Lambda) = \sup\big\{ d_\Lambda(V)^{-1} \, : \, \textrm{$V < \bR^d$ is $\Lambda$-rational} \big\}.
\]
It can readily be checked that $\alpha$ is a \emph{proper} function on $\xd$, and that for every compact set $\cC \subset \SL_d(\bR)$,
there is a constant $A_{\cC} > 0$ such that
\begin{equation}
\label{bndsalpha}
A_{\cC}^{-1} \alpha(\Lambda) \leq \alpha(g\Lambda) \leq A_{\cC} \alpha(\Lambda), \quad \textrm{for all $g \in \cC$ and $\Lambda \in \xd$}.
\end{equation}
Before we introduce Sobolev norm, we mention important properties of the $\alpha$-function in relation with Siegel transforms.
\vspace{0.2cm}

\begin{lemma}[\cite{Sch68}, Lemma 2]
\label{Siegelbnd}
If $f : \bR^d \ra \bR$ is a bounded function with bounded support, then
\[
\big|\widehat{f}(\Lambda)\big| \ll_{\supp (f)} \alpha(\Lambda) \|f\|_\infty , \quad \textrm{for all $\Lambda \in \xd$}.
\]
\end{lemma}
\vspace{0.2cm}

\noindent The following estimate is also well-known:
\vspace{0.2cm}

\begin{lemma}\label{l:alpha}
$\int_X \alpha^p \, d\mu<\infty$ for every $p<d$.
\end{lemma}

\vspace{0.4cm}

\noindent The following norms were introduced and studied by Einsieder, Margulis and Venkatesh \cite{EMV}. 
\begin{definition}[Sobolev norms]
Let $q$ be a positive integer. For $\varphi \in C^\infty_c(\xd)$, its \emph{Sobolev norm $S_q(\varphi)$ of order $q$} is defined as
\[
S_q(\varphi): = \sum_{|\eta| \leq q} \Big( \int_{\xd} |\alpha^d D_\eta \varphi|^2 \, d\md \Big)^{1/2}.
\]
\end{definition}

The explicit expression of the norm $S_q$ will not be important in our paper. Instead we shall use as black boxes, the following properties of
the norms, established in \cite{EMV} and in our previous paper \cite{BG2}.

\vspace{0.2cm}

\begin{proposition}[Section 5 in \cite{EMV}]
\label{prop_truncsob}
For all sufficiently large $q$, 
\vspace{0.1cm}
\begin{enumerate}
\item[(i)] $S_q(\varphi) \ll_q S_{q+1}(\varphi)$ and $\|\varphi\|_{L^\infty(X)} \ll_q S_q(\varphi)$ for all $\varphi \in C_c^\infty(\xd)$. \vspace{0.1cm}
\item[(ii)] for some $p \geq 1$, we have $S_q(\varphi_1 \varphi_2) \ll_q S_{q+p}(\varphi_1) S_{q+p}(\varphi_2)$, for all $\varphi_1, \varphi_2 \in C^\infty_c(\xd)$. \vspace{0.1cm}
\item[(iii)] there exists $\sigma_q > 0$ such that $S_q\big(\varphi\circ\aad(u)\big) \ll_q e^{\sigma_q \|u\|} S_q(\varphi)$, for all $u \in \bR^{k-1}$, 
where $\aad(u)$ is defined in \eqref{defb} and $\|\cdot\|$ is the $\ell^\infty$-norm on $\bR^{k-1}$.
\end{enumerate}
\end{proposition}
\vspace{0.2cm}

For our next proposition, we need some notation and preliminary results.  
First, we recall some further properties of the cut-off functions $\eta_L$ 
constructed in \cite{BG2}:

\vspace{0.2cm}

\begin{proposition}[Lemma 4.10 in \cite{BG2}]
\label{prop_suppetaL}
There exists a constant $c > 0$ such that 
\[
\supp \eta_L \subset \{ \alpha \leq cL \big\}, \quad \textrm{for all $L > 0$},
\]
and for all $q \geq 1$, $f \in C^\infty(\bR^d)$, and $L > 0$,
\[
\sup_{|\alpha| \leq q} \Big\|D_\alpha\big(\widehat{f}\,\eta_L\big)\Big\|_{L^\infty(X)} \ll_{\supp (f),q} L\,\|f\|_{C^q}.
\]
\end{proposition}

\vspace{0.2cm}

The following corollary concerning Sobolev norms of truncated Siegel transforms is now immediate.

\vspace{0.2cm}

\begin{corollary}
\label{cor_truncsob}
For all $q \geq 1$, $f \in C^\infty(\bR^d)$, and $L > 0$,
\[
S_q\big(\widehat{f} \, \eta_L\big) \ll_{\supp (f),q} L^{d+1} \|f\|_{C^q}.
\]
\end{corollary}

\vspace{0.2cm}

We also record the following corollary for future references. It is immediate from the inequalities in \eqref{bndsalpha} and the first part of Proposition \ref{prop_suppetaL}.

\vspace{0.2cm}

\begin{corollary}
\label{perturbetaL}
For every compact set $\cC \subset \SL_d(\bR)$, there is a constant $B_{\cC} > 0$ such that
\[
\eta_L \circ g \leq \chi_{\{\alpha \leq B_{\cC}L\}} \quad \textrm{for all $g \in \cC$ and $L > 0$}.
\]
\end{corollary}

\vspace{0.2cm}

Recall that $\abd \simeq \bR^{k-1}$ via the map $u \mapsto \aad(u)$ defined in \eqref{defb}. Let us throughout the rest of the
section denote by $\|\cdot\|$ the $\ell^\infty$-norm on $\bR^{k-1}$. The following theorem is a special case contained in \cite{BEG}
by Einsieder and the two authors obtained upon realizing $\xd$ as the quotient space $\SL_d(\bR)/\SL_d(\bZ)$. Roughly speaking, this theorem asserts that if $\varphi \in C_c^\infty(\xd)$, then the family $u \mapsto \varphi(\aad(u) \cdot)$ consists of "almost independent" random variables, 
at least if the $u$'s are far apart. 

\vspace{0.2cm}

\begin{theorem}[Theorem 1.1 in \cite{BEG}]
\label{thm_multmix}
For every $r \geq 2$, there exist $q_r \geq 1$ and $\delta_r > 0$ such that for all $q \geq q_r$,
$\varphi_1,\ldots,\varphi_r \in C^\infty_c(\xd)$, and $u^{(1)},\ldots,u^{(r)} \in \bR^{k-1}$,
\[
\left| \int_{\xd} \Big(\prod_{m=1}^r\varphi_{m}\circ \aad(u^{(m)})\Big) \, d\md - \prod_{m=1}^r \int_{\xd} \varphi_{m} \, d\md \right|
\ll_{r}
e^{-\delta_r \, \min_{j \neq k} \|u^{(j)}-u^{(k)}\|} \, \prod_{m=1}^r S_{q_r}(\varphi_m).
\]
\end{theorem}

\vspace{0.2cm}

Theorem 1.1 in \cite{BEG} is formulated for general $r$-tuples of elements in $G = \SL_d(\bR)$, and not just for $r$-tuples in $\abd$.
Furthermore, in the version in \cite{BEG}, the $\min_{i \neq j}$-expression is applied to differences with respect to an invariant Riemannian metric on $G$.
The restriction of any such metric to $\abd$ is quasi-isometric to the $\ell^\infty$-distance on $\bR^{k-1}$, and the resulting constants are assumed to have
been absorbed in $\delta_r$ and by the $\ll$-sign.

\subsection{Cumulants}
\label{subsec:cumulants}

We review the notion of cumulants, and a classical CLT-criterion due to Frechet and Shohat. In this subsection $(X,\mu)$ can be a general probability measure space. 

\begin{definition}[Cumulants]
\label{def_cumbrr}
Fix $r \geq 2$. Given $\varphi_1,\ldots,\varphi_r \in L^\infty(X)$, we define their \emph{cumulant $\cum_r(\varphi_1,\ldots,\varphi_r)$
of order $r$} by
\[
\cum_{[r]}(\varphi_1,\ldots,\varphi_r) := \sum_{\cP \in \mathfrak{P}_{[r]}} (-1)^{|\cP|-1} \prod_{I \in \cP} \Big( \int_X \prod_{i \in I} \varphi_i \, d\mu \Big),
\]
where $\mathfrak{\cP}_{[r]}$ denotes the set of partitions of the set $[r] = \{1,\ldots,r\}$. Given $\Phi \in L^\infty(X)$, we define its
\emph{$r$-cumulant $\cum_r(\Phi)$} by
\[
\cum_{r}(\Phi) := \cum_{[r]}(\Phi,\ldots,\Phi).
\]
\end{definition}

\begin{remark}
It is clear that $\cum_{[r]}$ is multi-linear in the functions $\varphi_1,\ldots,\varphi_r$, and if one of them is a constant function, then 
$\cum_{[r]}(\varphi_1,\ldots,\varphi_r) = 0$. In particular,
\[
\cum_{[r]}(\varphi_1,\ldots,\varphi_r) = \cum_{[r]}\left(\varphi_1 - \int_X \varphi_1 \, d\mu,\ldots,\varphi_r - \int_X \varphi_r \, d\mu\right)
\]
and
\[
\cum_r\left(\Phi-\int_X \Phi \, d\mu\right) = \cum_r(\Phi).
\]
Furthermore, the $2$-cumulant of $\Phi$ is just the $\mu$-variance of $\Phi$. 
\end{remark}

The main property of cumulants that makes them valuable to us in this paper is summarized in the following CLT-criterion
by Frechet and Shohat, which can be deduced from their results in \cite{FS}. It is essentially the classical method of moments 
tailored for (distributional) convergence to the normal distribution. 

\vspace{0.2cm}

\begin{proposition}[Frechet-Shohat's Cumulant Criterion]
\label{prop_cumulants}
Let $(\Psi_T)$ be a sequence of real-valued, bounded and measurable functions on $X$
such that 
\begin{itemize}
	\item $\int_X \Psi_T\, d\mu=0,$
	\item the limit $\sigma^2 := \lim_{T} \|\Psi_T\|^2_{L^2(X)}$ exists and is finite,
	\item $\cum_r(\Psi_T) \ra 0$ for all $r \geq 3$.
\end{itemize}
Then the $\md$-distributions of $\Psi_T$ converge in the sense of distribution to the Normal Law with mean zero and variance $\sigma^2$ (the case $\sigma = 0$ is interpreted as convergence in the sense of distributions to the Dirac measure at $0$).
\end{proposition}

\vspace{0.2cm}

In order to apply this proposition,
we have to analyze the cumulants $\cum_r(\Psi_T)$.
This task will be carried out in the next section.

\subsection{Estimating cumulants of order $r \geq 3$}

Let $\Psi_T$ be defined by \eqref{eq:psi}. Our goal is to show that under suitable additional conditions, 
$$
\cum_{r}(\Psi_T)\to 0\quad \quad \textrm{as $T \ra \infty$}
$$
for all $r\ge 3$. Since
$$
\cum_{r}(\Psi_T)
=V_T^{-r/2}\cum_{r}\left(\Phi_{T}-\int_{\xd} \Phi_{T} \, d\md\right) = V_T^{-r/2}\cum_{r}(\Phi_{T}), 
$$
this is equivalent to
\begin{equation}\label{eq:ccc}
\cum_{r}(\Phi_T)= o\left(V_T^{r/2}\right) \quad \textrm{as $T \ra \infty$}.
\end{equation}

Let us from now on fix $r \geq 3$. 
For each $r$-tuples $\bi = (i_1,\ldots,i_r) \in \mathcal{I}^r$, we set
\begin{align*}
Y_{T,\bi}: = Y_{T,i_1} \times \cdots \times Y_{T,i_r}\quad\hbox{and}\quad
\kappa_{T,\bi} : = \kappa_{T,i_1} \otimes \cdots \otimes \kappa_{T,i_r}, 
\end{align*}
and for $y = (y_1,\ldots,y_r) \in Y_{T,\bi}$, we set
\begin{align*}
Q_{T,\bi}(y) : = Q_{T,i_1}(y_1) \times \cdots \times Q_{T,i_r}(y_r).
\end{align*}
We write elements of $Q_{T,\bi}(y)$ as $\bu=(u^{(1)},\ldots, u^{(r)})$.
Using the multi-linearity of the cumulants, we see that
$\cum_{[r]}(\Phi_{T})$ can be written as
\[
 \sum_{\bi\in\mathcal{I}^r} \int_{Y_{T,\bi}} \sum_{\bu \in Q_{T,\bi}(y)} \cum_{[r]}\left(\varphi_{T,i_1}(\cdot,y_1)\circ \aad(u^{(1)}),\cdots,\varphi_{T,i_r}(\cdot,y_r)\circ \aad(u^{(r)})\right) \, d\kappa_{T,\bi}(y).
\]
We shall make the following additional assumptions regarding the data
defining the function $\Phi_T$. Throughout this section,
$\|\cdot\|$ denotes the $\ell^\infty$-norm on $\bR^{k-1}$ and $B(x,\gamma)$ the ball with respect to this norm.
\vspace{0.1cm}
\begin{enumerate}
\item[(II.a)] there exist finite sets $\widetilde{Q}_{T,i} \subset \bR^{k-1}$ 
satisfying:
\vspace{0.1cm}
\begin{itemize}
	\item for all $\gamma\ge 1$
\begin{equation}
\label{QTgrowth}
\big|\widetilde{Q}_{T,i} \cap B(u,\gamma)\big| \ll \gamma^{k-1},
\end{equation}
where the implicit constants are independent of $u$, $T$, and $i$.	\vspace{0.1cm} 
\item $\max_i |\tilde Q_{T,i}|\ll V_T$ with a parameter $V_T$ satisfying $V_T\to\infty$ as $T\to\infty$.
\end{itemize}
\vspace{0.2cm}

\item[(II.b)] there exist Borel maps $\beta_{T,i} : \bR^{k-1} \times Y_{T,i} \ra \bR^{k-1}$ such that 
$$
Q_{T,i}(y_i) = \beta_{T,i}\big(\widetilde{Q}_{T,i},y_i\big)
$$
satisfying:
\vspace{0.1cm}
\begin{itemize}
	\item there exist $c_1, c_2 > 0$, independent of $T$, such that
	for all $u,v \in \widetilde{Q}_{T,i}$,
	\begin{equation}
	\label{enumcond1}
	\min_{i,j} \inf_{y_i \in Y_{T,i}} \inf_{y_j \in Y_{T,j}} \big\|\beta_{T,i}(u,y_i) - \beta_{T,j}(v,y_j)\big\| \geq c_1 \|u-v\| - c_2, 
	\end{equation}
	
	\item there exist maps $\widetilde{\beta}_{T,i} : \bR^{k-1} \ra \bR^{k-1}$ such that for all $u \in \widetilde{Q}_{T,i}$,
	\begin{equation}
	\label{enumcond2}
	\sup_{T}  \sup_{y_i \in Y_{T,i}} \big\|\beta_{T,i}(u,y_i)-\widetilde{\beta}_{T,i}(u)\big\|  < \infty.
	\end{equation}
\end{itemize}

\item[(II.c)] For the functions $f_{T,i}$ from (I.b),
there exist Borel functions $h_{T,i} : \bR^d \times Y_{T,i} \ra [0,\infty)$ such that
\[
f_{T,i}\big(\aad(\beta_{T,i}(u,y_i))z,y_i\big) \leq h_{T,i}\big(\aad(\widetilde{\beta}_{T,i}(u))z,y_i\big)
\]
for all $u \in \widetilde{Q}_{T,i}$,  $y_i \in Y_{T,i}$, and $z \in \bR^d$.
We further assume that the family of the functions 
\[
H_{T,i}(z) := \int_{Y_{T,i}} h_{T,i}(z,y_i) \, d\kappa_{T,i}(y_i)
\] 
is uniformly bounded, and there exists a fixed compact set $\cK' \subset \bR^d$ such that 
$$
\supp(H_{T,i}) \subset \cK'
$$
for all $T$ and $i$.
\end{enumerate}

\begin{remark}
We note that the condition (I.c) from Subsection \ref{subsec:data}
follows immediately from condition (II.a) and the first part of condition (II.b). 
\end{remark}

With this new notation, we set
\[
\Xi_{r,T,\bi}(y) := 
\sum_{\bu \in \widetilde{Q}_{T,\bi}} \cum_{[r]}\left(\varphi_{T,i_1}(\cdot, y_{i_1})\circ\aad\big(\beta_{T,i_1}(u^{(1)},y_{i_1})\big),\cdots,
\varphi_{T,i_r}(\cdot, y_{i_r})\circ \aad\big(\beta_{T,i_r}(u^{(r)},y_{i_r})\big )\right),
\]
where $\widetilde{Q}_{T,\bi} := \widetilde{Q}_{T,i_1} \times \cdots \times \widetilde{Q}_{T,i_r}$. Then
\begin{equation}
\label{cumrexp}
\cum_{r}(\Phi_{T}) = \sum_{\bi\in\mathcal{I}^r} \int_{Y_{T,\bi}} \Xi_{r,T,\bi}(y) \, d\kappa_{T,\bi}(y).
\end{equation}

For $\gamma > 0$, we define the \emph{$r$-diagonal $\gamma$-neighborhood} $\Delta_r(\gamma)$ by
\[
\Delta_r(\gamma) := \left\{ (u^{(1)},\ldots,u^{(r)}) \in (\bR^{k-1})^r :\,\,  \|u^{(j)} - u^{(k)}\| \leq \gamma \;\;\hbox{for all $j,k$}\right\}.
\]
We split the sum defining $\Xi_{r,T,\bi}$ into
two subsums subdivided with respect to the set $\Delta_r(\gamma)$.
Namely, we choose a parameter $\gamma_{T,r}\to \infty$, which will be specified later, and write 
\[
\Xi_{r,T,\bi} = \Xi^{(1)}_{r,T,\bi} + \Xi^{(2)}_{r,T,\bi}\,,
\]
where $\Xi^{(1)}_{r,T,\bi}(y)$ denotes the sum over \emph{clustered} $r$-tuples
\begin{equation}
\label{defsum1}
\sum_{\bu \in \widetilde{Q}_{T,\bi} \cap \Delta_r(\gamma_{T,r})}
\cum_{[r]}\left(\varphi_{T,i_1}(\cdot, y_{i_1})\circ\aad\big(\beta_{T,i_1}(u^{(1)},y_{i_1})\big),\cdots,
\varphi_{T,i_r}(\cdot, y_{i_r})\circ \aad\big(\beta_{T,{i_r}}(u^{(r)},y_{i_r}) \big)\right)
,
\end{equation}
and $\Xi^{(2)}_{r,T,\bi}(y)$ denotes the sum over \emph{separated} $r$-tuples:
\begin{equation}
\label{defsum2}
\sum_{\bu \in \widetilde{Q}_{T,\bi} \cap \Delta_r(\gamma_{T,r})^c} 
\cum_{[r]}\left(\varphi_{T,i_1}(\cdot, y_{i_1})\circ\aad\big(\beta_{T,i_1}(u^{(1)},y_{i_1})\big),\cdots,
\varphi_{T,i_r}(\cdot, y_{i_r})\circ \aad\big(\beta_{T,{i_r}}(u^{(r)},y_{i_r}) \big)\right).
\end{equation}
The aim in the upcoming subsections is to find conditions on the parameters $\gamma_{T,r}$ and $L_T$ such that for every 
$\bi = (i_1,\ldots,i_r) \in \mathcal{I}^r$,
\begin{equation}
\label{1sum}
\int_{Y_{T,\bi}} \big|\Xi^{(1)}_{r,T,\bi}(y)\big| \, d\kappa_{T,\bi}(y) = o\left(V_T^{r/2}\right)\quad \hbox{as $T\to\infty$,}
\end{equation}
and
\begin{equation}
\label{2sum}
\sup_{y \in Y_{T,\bi}} \big|\Xi^{(2)}_{r,T,\bi}(y)\big| = o\left(V_T^{r/2}\right)
\quad \hbox{as $T\to\infty$.}
\end{equation}
Together with the assumption (I.a) in Subsection \ref{subsec:data}, 
these estimates imply \eqref{eq:ccc}. 

\subsubsection{Analysis of the separated tuples}

Now we prove the estimate \eqref{2sum} involving separated tuples. 
The crucial ingredient here is the estimates on higher-order correlations
(Theorem \ref{thm_multmix}), which allows to established an estimate on cumulants following our approach from \cite{BG1}. \\

We recall the estimate from  
Proposition \ref{prop_truncsob}(iii) that for every $q \geq 1$, there 
exists $\sigma_q > 0$ such that 
\[
S_q\big(\varphi\circ \aad(u)\big) \ll_q e^{\sigma_q \|u\|} S_q(\varphi)
\quad\quad  \hbox{for all $\varphi \in C_c^\infty(\xd)$ and $u \in \bR^{k-1}$.}
\]
 We may without
loss of generality assume that the map $q \mapsto \sigma_q$ is increasing. Furthermore, we may also assume that the map 
$r \mapsto \delta_r$ in Theorem \ref{thm_multmix} is decreasing. In particular, without loss of generality we can assume that
\begin{equation}\label{eq:1}
\delta_r < r \sigma_q, \quad \textrm{for all $q,r \geq 1$}.
\end{equation}
The following lemma is a corollary of the main technical results from our work \cite{BG1}. 
\vspace{0.2cm}

\begin{lemma}
\label{thm_BG}
There is an integer $q_r \geq 1$, such that for every integer $q > q_r$, there exists a constant $c_{r,q} > 0$, 
with the property that for every $\gamma > 0$ and for all $\varphi_1,\ldots,\varphi_r \in C_c^\infty(\xd)$ and $u^{(1)},\ldots,u^{(r)} \in \bR^{k-1}$ 
with
\[
\max_{j,k} \|u^{(j)}-u^{(k)}\| > c_{r,q} \gamma, 
\]
we have
\[
\left|\cum_{[r]}\left(\varphi_1\circ\aad(u^{(1)}),\ldots,\varphi_r\circ\aad(u^{(r)})\right)\right| \ll_{r,q} e^{-\gamma} \, \prod_{j=1}^r S_{q}(\varphi_j).
\]
\end{lemma}

\vspace{0.2cm}

\begin{proof}
The proof  follows the argument in  \cite[Sec. 6.4]{BG1}. 
Let us fix $r \geq 2$, $\gamma > 0$ and an integer $q \geq 1$. 
We define parameters $\beta_0=0$, $\beta_1$,\ldots, $\beta_r$ 
recursively by $\beta_{j+1} \delta_r - 3 r  \sigma_q \beta_j = \gamma$. 
Then because of \eqref{eq:1},
\[
0<\beta_1 < 3\beta_1 < \beta_2 < \cdots < \beta_{r-1} <  3 \beta_{r-1} < \beta_r.
\]
It is also clear from the recursive definition that 
\begin{equation}
\label{induction}
\beta_r \leq c_{r,q} \, \gamma 
\end{equation}
for a constant $c_{r,q}>0$.
Combining \cite[Prop. 6.1]{BG1} and \cite[Prop. 6.2]{BG1},
we conclude that there is an integer $q_r$ such that if $q > q_r$, then 
\[
\left|\cum_{[r]}\left(\varphi_1\circ \aad(u^{(1)}),\ldots,\varphi_r\circ \aad(u^{(r)}) \right)\right| \ll_{r,q} e^{-\gamma} \, \prod_{j=1}^r S_{q}(\varphi_j),
\]
for all $u^{(1)},\ldots,u^{(r)}\in \bR^{k-1}$ such that $\max_{j,k} \|u^{(j)} - u^{(k)}\| > \beta_r$. Together with \eqref{induction}, this proves the lemma.
\end{proof}

Now we apply Lemma \ref{thm_BG} to estimate $\Xi^{(2)}_{r,T,\bi}$, and deduce a criterion for \eqref{2sum}. From now on $q_r$ denotes the integer from Lemma \ref{thm_BG}.

\vspace{0.2cm}

\begin{proposition}
\label{cor_septerms}
Suppose that the parameters $L_T$ and $\gamma_{T,r}$ are chosen so that
for some $q > q_r$, 
\begin{equation}\label{eq:cond2}
L_T^{r(d+1)} \,  V_T^{r/2} \, e^{-c_1 \gamma_{T,r}/c_{r,q}} \,  M^r_{T,q}\to 0,\quad\hbox{as $T\to\infty$},
\end{equation}
where $c_1$ is the positive constant in condition (II.b),
and $c_{r,q}$ is given by Lemma \ref{thm_BG}. 
Then, for every $\bi = (i_1,\ldots,i_r)\in\mathcal{I}^r$,
\[
\sup_{y \in Y_{T,\bi}} \big|\Xi^{(2)}_{r,T,\bi}(y)\big| = o\left(V_T^{r/2}\right), \quad \textrm{as $T \ra \infty$}.
\]
\end{proposition}

\vspace{0.2cm}

\begin{proof}
We first note 
that if $(u^{(1)},\ldots,u^{(r)})$ belongs to $\widetilde{Q}_{T,\bi} \cap \Delta_r(\gamma_{T,r})^c$, then by condition (II.b),
there exist $i_m,i_n \in \cI$ such that 
\[
 \big\|\beta_{T,i_m}(u^{(m)},y_{i_m}) - \beta_{T,i_n}(u^{(n)},y_{i_n})\big\| 
> c_1 \gamma_{T,r} - c_2,
\]
for all $y_{i_m} \in Y_{T,i_m}$ and $y_{i_n} \in Y_{T,i_n}$.
Applying Lemma \ref{thm_BG} with $\gamma$ defined by 
$c_1 \, \gamma_{T,r} - c_2 = c_{r,q} \gamma,$ we deduce that
\[
\left|\cum_{[r]}\left(\varphi_{T,i_1}\circ \aad\big(\beta_{T,i_1}(u^{(1)},y_{i_1})\big),\cdots,
\varphi_{T,i_r}\circ \aad\big(\beta_{T,i_r}(u^{(r)},y_{i_r})\big)\right)\right| 
\]
is estimated by
\[
\ll_{r,q} 
 e^{-c_1 \gamma_{T,r}/c_{r,q}} \prod_{m=1}^r S_q(\varphi_{T,i_m}),
\]
where we in the last $\ll$-sign have absorbed the $e^{-c_2/c_{r,q}}$-factor. We recall that
\[
\varphi_{T,i}(\Lambda,y_i) = \widehat{f}_{T,i}(\Lambda,y_i) \eta_{L_T}(\Lambda).
\]
By Corollary \ref{cor_truncsob}, 
\[
S_{q}\big(\varphi_{T,i}(\cdot,y_i)\big) \ll_{\cK,q} L_T^{d+1} \, \big\|f_{T,i}(\cdot,y_i)\big\|_{C^q},
\]
where $\cK \subset \bR^d$ is a fixed compact set which contains all of the supports of the functions $x \mapsto f_{T,i}(x,y_i)$ as $y_i$ ranges
over $Y_{T,i}$. We conclude that
\begin{align*}
\sup_{y \in Y_{T,\bi}} \big|\Xi^{(2)}_{r,T,\bi}(y)\big| &\ll_{r,\cK,q} \left(\prod_{m=1}^r |\widetilde{Q}_{T,i_m}| \right) \, e^{-c_1 \gamma_{T,r}/c_{r,q}} \, L_T^{r(d+1)} M_{T,q}^r\\
&= V_T^r \, e^{-c_1 \gamma_{T,r}/c_{r,q}} \, L_T^{r(d+1)} M_{T,q}^r.
\end{align*}
This implies the proposition.
\end{proof}

\subsubsection{Analysis of the clustered tuples}

Next, we deal with the clustered tuples. Our analysis here is one of the main novelties of this paper. 
We stress that we do \emph{not} assume that the maps $T \mapsto \|f_{T,i}\|_\infty$ 
are bounded (otherwise, our analysis could have been carried out as in \cite{BG1}).
This is also where the assumption (I.c) becomes crucial.
This condition says 
roughly that the $\kappa_{T,i}$-integrals of $f_{T,i}$ are bounded functions. The main purpose of this subsection is to explain how this "bounded on average"-condition can be used to derive \eqref{1sum}. 
\vspace{0.2cm}

\begin{proposition}
\label{prop_new}
Suppose that the parameters $L_T$ and $\gamma_{T,r}$ satisfy for some $\varepsilon>0$,
\begin{equation}\label{eq:cond3}
L_T^{r-d+\varepsilon}\, V_T^{1-r/2}\, \gamma_{T,r}^{(r-1)(k-1)}   \to 0,\quad\hbox{as $T\to\infty$}.
\end{equation}
Then,
\[
\int_{Y_{T,\bi}} \big|\Xi^{(1)}_{r,T,\bi}(y)\big| \, d\kappa_{T,\bi}(y) =o\left(V_T^{r/2}\right), \quad \textrm{as $T \ra \infty$}.
\]
\end{proposition}

\vspace{0.2cm}

\begin{proof}
Expanding the definition of the cumulant in \eqref{defsum1}, we deduce that
\begin{equation}\label{eq:xi1}
\big|\Xi^{(1)}_{r,T,\bi}(y)\big|\ll_r \max_{\cP}
\sum_{\bu \in \widetilde{Q}_{T,\bi} \cap \Delta_r(\gamma_{T,r})}
\prod_{I\in \cP}\int_X\left(\prod_{k\in I} \varphi_{T,i_k}\left(\aad\big(\beta_{T,i_k}(u^{(k)},y_{i_k})\big)\Lambda, y_{i_k}\right)\right)	\, d\mu(\Lambda).
\end{equation}	
We recall that
\[
\varphi_{T,i}(\Lambda,y_i) = \widehat{f}_{T,i}(\Lambda,y_i) \eta_{L_T}(\Lambda).
\]
By condition (II.c), there exist Borel functions $h_{T,i} : \bR^d \times Y_{T,i} \ra [0,\infty)$ such that
\[
{f}_{T,i}\big(\aad\big(\beta_{T,i}(u,y_i)\big)z,y_i\big) \leq {h}_{T,i}\big(\aad\big(\widetilde{\beta}_{T,i}(u)\big)z,y_i\big),
\]
for all $u \in \widetilde{Q}_{T,i}$, $z\in\mathbb{R}^d$, and $y_i\in Y_{T,i}$. Hence, setting
\[
h(z) := \sup_{T,i}\int_{Y_{T,i}} h_{T,i}(z,y_i) \, d\kappa_{T,i}(y_i),
\]
we deduce that
\begin{align*}
\int_{Y_{T,i}} \widehat{f}_{T,i}\big(\aad\big(\beta_{T,i}(u,y_i)\big)\Lambda,y_i\big)
\, d\kappa_{T,i}(y_i)&=
\sum_{z\in\Lambda\backslash \{0\}}\int_{Y_{T,i}} {f}_{T,i}\big(\aad\big(\beta_{T,i}(u,y_i)\big)z,y_i\big)
\, d\kappa_{T,i}(y_i)\\
&\le 
 \widehat{h}\big(\aad\big(\widetilde{\beta}_{T,i}(u)\big)\Lambda\big).
\end{align*}
We recall that according condition (II.c), the function $h$ 
is uniformly bounded and its support is contained in a fixed compact
set. In particular, it follows from Lemma \ref{Siegelbnd} that 
\begin{equation}\label{eq:Gb}
\widehat{h}(\Lambda) \ll \alpha(\Lambda), \quad \textrm{for all $\Lambda \in \xd$}.
\end{equation}
By condition (II.b), there is a fixed compact set $\cC \subset \abd$ such that
\[
\aad\left(\beta_{T,i}(u,y_i) - \widetilde{\beta}_{T,i}(u)\right) \in \cC, \quad \textrm{for all $u \in \widetilde{Q}_{T,i}$, $y_i \in Y_{T,i}$, and $T>0$}.
\]
By Corollary \ref{perturbetaL}, there is a constant $B=B({\cC}) > 0$ such that
\[
\eta_{L_T} \circ g \leq \chi_{ \{\alpha \leq B\, L_T \}} \quad \textrm{for all $T$ and $g \in \cC$},
\]
whence
\begin{eqnarray*}
\eta_{L_T}\big(\aad\big(\beta_{T,i}(u,y_i)\big) \Lambda\big) 
&=&  \eta_{L_T}\left(\aad\big(\beta_{T,i}(u,y_i) - \widetilde{\beta}_{T,i}(u)\big) \aad\big(\widetilde{\beta}_{T,i}(u)\big) \Lambda \right) \\
&\leq & 
\chi_{ \{ {\alpha} \leq B\, L_T \}}\left(\aad\big(\widetilde{\beta}_{T,i}(u)\big) \Lambda\right).
\end{eqnarray*}
Combining the above estimates, we conclude that
\[
\int_{Y_{T,i}} \varphi_{T,i}\left(a\big (\beta_{T,i}(u,y_i)\big) \Lambda ,y_i\right) \, d\kappa_{T,i}(y_i)\le \psi_T\left(a\big(\widetilde{\beta}_{T,i}(u)\big)\Lambda\right),
\]
where $\psi_T$ is defined by 
\begin{equation}\label{eq:ht} 
\psi_T(\Lambda):= \widehat{h}(\Lambda)\, \chi_{ \{ \alpha  \leq B\, L_T \}}(\Lambda),\quad\hbox{for $\Lambda\in X$.}
\end{equation}
Therefore, we deduce from \eqref{eq:xi1} that
\begin{equation}\label{eq:xi2}
\int_{Y_{T,\bi}} \big|\Xi^{(1)}_{r,T,\bi}(y)\big| \, d\kappa_{T,\bi}(y) \ll_r \max_{\cP}
\sum_{\bu \in \widetilde{Q}_{T,\bi} \cap \Delta_r(\gamma_{T,r})}
\prod_{I\in \cP}\int_X\left(\prod_{k\in I} \psi_{T}\left(\aad\big(\tilde\beta_{T,i_k}(u^{(k)})\big)\Lambda\right)\right)	\, d\mu(\Lambda).
\end{equation}	
We observe that it follows from \eqref{eq:Gb}, \eqref{eq:ht}, and Lemma \ref{l:alpha} that
\begin{equation}\label{eq:h1}
\sup|\psi_T|=O(L_T)\quad\hbox{and}\quad \|\psi_T\|_{L^p(X)}=O_p(1)\;\;\hbox{for $p<d$.}
\end{equation}
In particular, it also follows that for $p\ge d$, 
$$
\|\psi_T\|_{L^p(X)}=O_{p,q}\left(L_T^{1-q/p}\right)\quad\hbox{for all $q<d$.}
$$
According to the general H\"older inequality,
for exponents $p_k\in (1,\infty]$ satisfying $\sum_{k} 1/p_k=1$,
$$
\int_X\left(\prod_{k\in I} \psi_{T}\left(\aad\big(\tilde\beta_{T,i_k}(u^{(k)})\big)\Lambda\right)\right)	\, d\mu(\Lambda)\le \prod_{k\in I} \left\|\psi_T\circ \aad\big(\tilde\beta_{T,i_k}(u^{(k)})\big)\right\|_{L^{p_k}(X)}= \prod_{k\in I} \|\psi_T\|_{L^{p_k}(X)}.
$$
Therefore, when $|I|<d$, 
$$
\int_X\left(\prod_{k\in I} \psi_{T}\left(\aad\big(\tilde\beta_{T,i_k}(u^{(k)})\big)\Lambda\right)\right)	\, d\mu(\Lambda)=O(1),
$$
and when $|I|\ge d$,
$$
\int_X\left(\prod_{k\in I} \psi_{T}\left(\aad\big(\tilde\beta_{T,i_k}(u^{(k)})\big)\Lambda\right)\right)	\, d\mu(\Lambda)=O_\varepsilon \left(L_T^{|I|-d+\varepsilon}\right)\quad\hbox{for all $\varepsilon>0$.}
$$
We conclude that for every partition $\cP$,
$$
\prod_{I\in \cP}\int_X\left(\prod_{k\in I} \psi_{T}\left(\aad\big(\tilde\beta_{T,i_k}(u^{(k)})\big)\Lambda\right)\right)	\, d\mu(\Lambda)=O_\varepsilon \left(L_T^{r-d+\varepsilon}\right),
$$
and from \eqref{eq:xi1},
$$
\int_{Y_{T,\bi}} \big|\Xi^{(1)}_{r,T,\bi}(y)\big| \, d\kappa_{T,\bi}(y)
\ll_{r,\varepsilon} \left|\widetilde{Q}_{T,\bi} \cap \Delta_r(\gamma_{T,r})\right|\, L_T^{r-d+\varepsilon}.
$$
Since 
\[
\left|\widetilde{Q}_{T,\bi} \cap \Delta_r(\gamma_{T,r})\right| \leq
\sum_{u\in \widetilde{Q}_{T,i_1}}  \prod_{k=2}^{r} \left|\widetilde{Q}_{T,i_k} \cap \{v:  \|v-u\| \leq \gamma_{T,r} \}\right|,
\]
it follows from condition (II.a) that
\[
\left|\widetilde{Q}_{T,\bi} \cap \Delta_r(\gamma_{T,r})\right| \ll V_T \gamma_{T,r}^{(k-1)(r-1)},
\]
whence
\[
\int_{Y_{T,\bi}} \left|\Xi^{(1)}_{r,T}(y)\right| \, d\kappa_{T,i}(y) \ll_{r,\varepsilon} V_T \gamma_{T,r}^{(k-1)(r-1)} L_T^{r-d+\varepsilon} ,
\]
for all $\varepsilon>0$, which implies the assertion of the proposition. 
\end{proof}

\subsection{Main result}
\label{subsec:summary}

In this section, we finally prove convergence in distribution of the functions
\[
\Upsilon_T(\Lambda)= V_T^{-1/2}\left(\widehat{F}_T(\Lambda) - \int_{\xd} \widehat{F}_T \, d\md\right),
\]
where
$$
\widehat{F}_T(\Lambda) = \sum_{i\in\mathcal{I}} \int_{Y_{T,i}} \Big( \sum_{a \in Q_{T,i}(y_i)} \widehat{f}_{T,i}(a\Lambda,y_i) \Big) \, d\kappa_{T,i}(y_i).
$$
We recall that the data in this formula satisfy the conditions (I.a)--(I.c)
and (II.a)--(II.c). We further put an additional condition on the norms of
the functions $f_{T,i}$, using the notation introduced in \eqref{defM0_0}--\eqref{defM0}. \\

The main result of Section \ref{sec:general} is the following theorem:

\vspace{0.2cm}

\begin{theorem}\label{thm_criterion}
Suppose that 
\vspace{0.1cm}
\begin{itemize}
	\item There exists $\theta_0 > 0$ such that 
  $$
M_{T} = O\left(V_T^{\theta_0}\right).
$$
	
	\item For $q \geq 1$, there exists $\theta_q > 0$ such that 
	$$
	M_{T,q} = O\left(V_T^{\theta_q}\right).
	$$
	\item The limit 
	$$
	\sigma:=\lim_{T\to\infty} \|\Upsilon_T\|_{L^2(X)}
	$$ exists and is finite.
\end{itemize}
\vspace{0.1cm}
If $d > 4(1 + \theta_0)$, then the functions $\Upsilon_T$ on $(X,\mu)$
 converge in distribution to the Normal Law with variance $\sigma$.
\end{theorem}
 
 \vspace{0.2cm}
 
\begin{proof}
We shall use Proposition \ref{prop_cumulants}. 	
We recall that by Lemma \ref{lemma_trunc}, the functions $\widehat{F}_T$
can be approximated by functions 
\[
\Phi_{T}(\Lambda) := \sum_{i\in \cI} \int_{Y_{T,i}} \Big( \sum_{a \in Q_{T,i}(y_i)} \varphi_{T,i}(a\Lambda,y_i) \Big) \, d\kappa_{T,i}(y_i),
\]
so that
$$
\left\|\widehat{F}_T-\Phi_T\right\|_{L^2(X)}=o\left(V_T^{1/2}\right).
$$
This implies that the functions
\[
\Psi_T(\Lambda)= V_T^{-1/2}\left(\Phi_T(\Lambda) - \int_{\xd} \Phi_T \, d\md\right)
\]
satisfy
$$
\Big\|\Upsilon_T-\Psi_T\Big\|_{L^2(X)}\to 0.
$$
Then, in particular, 
$\lim_{T\to\infty} \|\Psi_T\|_{L^2(X)}=\sigma$.
It also follows that if $\Psi_T$ converges in distribution to the 
Normal Law, so does $\Upsilon_T$. Hence, it remains to verify
that the conditions of Proposition \ref{prop_cumulants} hold from the functions $\Psi_T$, namely, that
$$
\cum_{r}(\Psi_T)=V_T^{-r/2}\cum_{[r]}(\Phi_T)\to 0\quad\hbox{for all $r\ge 3$.}
$$
Since the later cumulant can be expressed as \eqref{cumrexp},
this will follow from Proposition \ref{cor_septerms} and Proposition \ref{prop_new}. 

Now it remains to choose the parameters 
$L_T$ and $\gamma_{T,r}$ so that the conditions in Lemma \ref{lemma_trunc}, Proposition \ref{cor_septerms}, and Proposition \ref{prop_new} are satisfied. To 
do this, we shall take 
\begin{equation}
\label{chooseLTgammaTr}
L_T = V_T^{\rho} \qand \gamma_{T,r} = S_{r} \log V_T,
\end{equation}
where $\rho$ and $S_r$ are positive real numbers, which will be chosen later. 
The condition \eqref{eq:cond1} in Lemma \ref{lemma_trunc} is satisfied if $\rho$ is chosen so that for some $\eps>0$
\[
V_T^{\rho(-d/2+1+\eps) + 1/2 +\theta_0} \ra 0,
\]
or equivalently, if
\begin{equation}
\label{ass1}
\rho > \frac{1 + 2\theta_0}{d-2-2\eps}.
\end{equation}
We write $q_r$ for the index introduced in Lemma \ref{cumrexp} and fix an integer $q > q_r$. 
The condition \eqref{eq:cond2} in Proposition \ref{cor_septerms} is satisfied if 
\[
V_T^{\rho r(d+1) + r/2- \frac{c_1 S_r}{c_{r,q}} +  r \theta_q} \ra 0,
\]
which can always be arranged by choosing $S_r$ large enough, depending on $r,\rho, d$. Finally, the condition \eqref{eq:cond3} in Proposition \ref{prop_new}
is satisfied if we choose the constants $\rho$ and $S_r$ such that 
for some $\varepsilon>0$,
\[
V_T^{\rho (r-d+\varepsilon) + 1 -r/2} (S_r\, \log V_T)^{(r-1)(k-1)} \ra 0.
\]
This holds provided that
\begin{equation}
\label{ass2}
\rho (r-d+\varepsilon)< r/2 -1.
\end{equation}
Hence, it is sufficient to choose $\rho$ so that both \eqref{ass1} and \eqref{ass2} hold for all $r\ge 3$.
This is possible provided that 
\[
\frac{1 + 2\theta_0}{d-2-2\eps} < \rho \le \frac{1}{2}.
\]
Since $\eps > 0$ is arbitrary, this argument works provided that 
$d > 4 + 4 \theta_0$.
\end{proof}

\begin{remark}
In order to proceed with the proof above it is sufficient to have that 
\begin{equation}\label{eq:ss1}
\Big\|\Upsilon_T-\Psi_T\Big\|_{L^1(X)}\to 0
\end{equation}
and 
\begin{equation}\label{eq:ss2}
\lim_{T\to\infty} \big\|\Psi_T\big\|_{L^2(X)}=\sigma.
\end{equation}
According to Lemma \ref{lemma_trunc}, condition \eqref{eq:ss1} holds under assumption \eqref{eq:cond1_1}. This assumption is weaker than \eqref{eq:cond1}, so that we can replace \eqref{ass1} by the assumption
\begin{equation}
\label{ass1_1}
\rho > \frac{1 + 2\theta_0}{2d-2-2\eps}.
\end{equation}
Then the argument can be carried out when $d > 2(1 +  \theta_0)$, provided that we 
can establish \eqref{eq:ss2} independently.
\end{remark}

\section{Proof of the main theorem }
\label{sec:smoothapprox}

In this section, we prove our main theorem (Theorem \ref{main}).
We recall that our goal is to analyze the lattice counting function for the domains  
\begin{equation}
\label{OmegaT}
\Omega_T =\Omega_T(I,B)= \big\{ z \in \bR^{d}: \, N_L(z) \in I, \enskip \td_L(z) \in B\;\;\hbox{and}\;\; 0<\|L_1(z)\|,\ldots, \|L_k(z)\| <T \big\}.
\end{equation}
Ultimately, we will construct an approximation of the characteristic function $\chi_{\Omega_T}$ by functional averages of the
form \eqref{FT0} and show that these functional averages satisfy the assumptions of Theorem \ref{thm_criterion}, so that 
Theorem \ref{main} will be a consequence of Theorem \ref{thm_criterion}.
This is a tedious and rather technical task, so it
might be beneficial for the reader to first take a look in Subsection \ref{subsec:punchline}, where the main objects of the section 
are summarized, and the most important verifications are indexed.  

\subsection{A basic reduction}

Let $L_j : \bR^{d} \ra \bR^{d_j}$ with $j=1,\ldots,k$, $I\subset (0,\infty)$, and $B \subset \bS_{\bd}$ be the objects defining the sets $\Omega_T$.
We also consider the basic domains 
\begin{equation}
\label{defOT}
\Omega^0_T (I,B) := \big\{ z \in \bR^{d} \, \mid \, N(z) \in I, \enskip \td(z) \in B \qen 0<\|z_1\|,\ldots, \|z_k\| < T \big\},
\end{equation}
where
\begin{equation}
\label{defNdThetad}
N(z): = \prod_{j=1}^k \|z_j\|^{d_j} \qand \td(z): = \Big( \frac{z_1}{\|z_1\|}, \ldots, \frac{z_k}{\|z_k\|} \Big).
\end{equation}
Then $\Omega_T = L^{-1}(\Omega^0_T)$ 
for the invertible linear map $L=(L_1,\ldots,L_k)$.
Let us write $L=cL_0$ with $c\in \bR^\times $ and $\det(L_0)=1$. Then
\[
\Omega_T = L_0^{-1}\big(\sgn(c)|c|^{-1/d}\Omega^0_T(I,B)\big) = L_0^{-1}\big(\Omega^0_T(|c|^{-1}I,\sgn(c)B)\big).
\]
Therefore, for any lattice $\Lambda$,
$$
\big| \Lambda \cap \Omega_T|=\big| L_0(\Lambda) \cap  \Omega^0_T(|c|^{-1}I,\sgn(c)B)\big|,
$$
and
$$
\hbox{Vol}\big(\Omega^0_T(|c|^{-1}I,\sgn(c)B)\big)=\hbox{Vol}\big(\Omega^0_T(I,B)\big).
$$
Since the measure on the space of lattice is invariant under $L_0$,
it is sufficient to analyze the distribution of the function
$\Lambda\mapsto | \Lambda \cap \Omega^0_T| - \hbox{Vol}(\Omega^0_T)$.

\medskip

{\it From now on we assume that the sets $\Omega_T=\Omega_T(I,B)$ are defined by \eqref{defOT}, where $I$ is a non-empty bounded interval in $(0,\infty)$, 
and $B$ is a Borel subset of $\bS_{\bd}$ with positive measure.}

\subsection{A coodinate system}
\label{subsec:changeofvar}

The sets $\Omega_T$ are more conveniently studied in a different coordinate system which we now introduce. 
We use notations
$$
\bR_{*}^{\bd}:={\prod}_{j=1}^k \bR^{d_j}\backslash \{0\}\quad\quad\hbox{and} \quad\quad  \bS_{\bd} := {\prod}_{j=1}^k S^{d_j-1}.
$$
Let 
\begin{equation}
\label{defphi}
\pid : \bR_{*}^{\bd} \longrightarrow \bR^{k-1} \times \bR \times \bS_{\bd}: \enskip z \mapsto \big(\ud(z),\sd(z),\td(z)\big),
\end{equation}
where 
\begin{align*}
\ud(z) & := \big(\log \|z_1\|,\ldots,\log \|z_{k-1}\| \big),\\
\sd(z) &:= \log N(z) = \sum_{j=1}^k d_j \log \|z_j\|,  \\
\td(z) &:= \Big( \frac{z_1}{\|z_1\|}, \ldots, \frac{z_k}{\|z_k\|} \Big).
\end{align*}
It is readily checked that the map $\pid$ is equivariant
with respect to the group $A$ defined in \eqref{defb} in the following sense:
\begin{equation}
\label{equiad}
\pid(\aad(u)z) = (\ud(z) + u,\sd(z),\td(z)), \quad \textrm{for all $u\in\bR^{k-1}$ and $x \in \bR_*^{\bd}$},
\end{equation}
and that the inverse map $\pid^{-1}$ is given by
\begin{equation}
\label{defpsi}
\pid^{-1}(u,s,\xi) = \left(e^{u_1}\xi_1,\ldots,e^{u_{k-1}} \xi_{k-1},e^{\big(s-\sum_{j=1}^{k-1} d_j u_j\big)/d_{k}} \xi_k \right).
\end{equation}
If one computes the Jacobian of this inverse map, the following lemma emerges:
\vspace{0.2cm}
\begin{lemma}
\label{lemma_volumeint}
 For every bounded Borel function $f : \bR^{k-1} \times \bR \times \bS_{\bd} \ra \bR$ with bounded support, 
\[
\int_{\bR_*^{\bd}} f(\pid(z)) \, dz = \frac{1}{d_k} \int_{\bS_{\bd}} \int_{\bR} \Big( \int_{\bR^{k-1}} f(u,s,\xi) \, du \Big) \, e^s \, ds \, d\kd(\xi).
\]
Here $dz$ denote the volume element on $\bR^d$ which assigns volume one to the unit cube, $du$ is the volume element on $\bR^{k-1}$
such that the unit cube in $\bR^{k-1}$ has volume one,
and the measure $\kappa$ is defined in \eqref{eq:k}.
\end{lemma}
\vspace{0.2cm}

Let us now write out the set $\Omega_T$ in $(u,s,\xi)$-coordinates. We define
\[
\Delta_T := \pid(\Omega_T) \subset \bR^{k-1} \times \bR \times \bS_{\bd},
\]
and given a point $z$ in $\bR_*^{\bd}$, we set
\[
u=\ud(z) = (u_1,\ldots,u_{k-1}),\quad s = \sd(z), \quad \xi = \td(z).
\]
Then $z \in \Omega_T$ if and only if 
\[
s \in \log I,\quad \xi \in B,\quad u_1  < \log T,\ldots, u_{k-1}< \log T,\quad s - \sum_{j=1}^{k-1} d_j u_j < d_k \log T.
\]
We now set $v_j = u_j - \log T$ for $j = 1,\ldots,k-1$. Then, 
the above conditions on $u$ are equivalent to
\begin{equation}
\label{def_v}
v_1,\ldots,v_{k-1} < 0\quad\quad\hbox{and}\quad\quad
\sum_{j=1}^{k-1} d_jv_j > -(d\log T - s).
\end{equation}
For $s < d \log T$, we let $\delta_T(s)$ denote the diagonal $(k-1) \times (k-1)$-matrix whose diagonal elements 
$\delta_{T,j}(s)$ are given by
\[
\delta_{T,j}(s): = \frac{d\log T-s}{d_j}, \quad \textrm{for $j=1,\ldots,k-1$}.
\]
We note that since the interval $I$ is bounded, the inequality $s < d \log T$ is satisfied for all $x \in \Omega_T(I,B)$ when $T> e^{\sup(I)/d}$.
Then \eqref{def_v} can be re-written as
\[
\min_j \, \delta_{T,j}(s)^{-1} v_j < 0\quad\hbox{and}\quad \sum_{j=1}^{k-1} \delta_{T,j}(s)^{-1}v_j > -1,
\]
Let
\begin{equation}\label{eq:s1}
\cS_1 := \left\{ (w_1,\ldots,w_{k-1}) \in \bR^{k-1} \, : \, w_1,\ldots,w_{k-1} < 0 \qen {\sum}_{j=1}^{k-1} w_j > -1 \right\}
\end{equation}
and
$$
v_T: = (\log T,\ldots,\log T).
$$ 
We conclude that 
\begin{equation}\label{eq:delta}
\Delta_T=\pi(\Omega_T)=\big\{(u,s,\xi)\, :\,\, s \in \log I, \quad \xi \in B, \quad u \in \delta_T(s)\cS_1 + v_T\big\}
\end{equation}
when $T> e^{\sup(I)/d}$.

\subsection{Volume and variance computations}\label{subsec:calcsigma}

The above parametrization of $\Omega_T$ leads, in particular, to an an easy computation of its volume, and the mean and  the variance of the Siegel transforms 
$\widehat\chi_{\Omega_T}$.

\vspace{0.2cm}

\begin{lemma}
	\label{lemma_volume}
There exists a polynomial 
	$P_{I,B}$ such that
	\[
	P_{I,B}(t) = c_{k-1}(I,B) t^{k-1} + O(t^{k-2}),
	\]
	where 
	\[
	c_{k-1}(I,B) = \frac{d^{k-1}}{d_1\cdots d_{k}} \Leb(I) \, \Vol_{k-1}(S_1) \, \kd(B),
	\]
	such that 
	$$
	\Vol\big(\Omega_T(I,B)\big) = P_{I,B}(\log T),
	$$
	for all $T> e^{\sup(I)/d}$. 
\end{lemma}

\vspace{0.2cm}

\begin{proof}
It follows from  \eqref{eq:delta} and Lemma \ref{lemma_volumeint} that 
\begin{eqnarray*}
\Vol(\Omega_T) 
&=&
\frac{\kd(B)}{d_k} \int_{\log I} \Vol_{k-1}\big(\delta_T(s)S_1 + v_T\big) \, e^s \, ds \\
&=&
\frac{\kd(B)}{d_k} \, \Vol_{k-1}(S_1) \int_{\log I} \frac{(d\log T-s)^{k-1}}{d_1\cdots d_{k-1}} \, e^s \, ds .
\end{eqnarray*} If we expand the inner parenthesis and integrating term-wise, we
deduce that $\Vol(\Omega_T) = P_{I,B}(\log T)$ for the polynomial
\[
P_{I,B}(t) = \frac{\kd(B)}{d_k}\, \Vol_{k-1}(S_1) \int_{\log I} \frac{(d t-s)^{k-1}}{d_1\cdots d_{k-1}} \, e^s \, ds .
\]
The leading term of this polynomial is $c_{k-1}(I,B)t^{k-1}$ with
\[
c_{k-1}(I,B) =  \frac{d^{k-1} }{d_1 \cdots d_k} \kd(B)\Vol_{k-1}(\cS_1) \int_{\log I} e^s \, ds = \frac{d^{k-1}}{d_1\cdots d_{k}} 
\kd(B)\, \Vol_{k-1}(\cS_1) \, \Leb(I),
\]
which finishes the proof of the lemma.
\end{proof}

For \eqref{eq:siegel2}, we also obtain that 
$$
\int_X \widehat \chi_{\Omega_T}\, d\mu= P_{I,B}(\log T)
= c_{k-1}(I,B) (\log T)^{k-1} + O\big((\log T)^{k-2}\big).
$$

\medskip

To compute the variance of the Siegel transform, we need the following

\vspace{0.2cm}

\begin{theorem}[Rogers' mean-square value theorem, \cite{Rog}]
	\label{thm_rogers}
	Let $d \geq 3$ and let $f : \bR^d \ra \bR$ be a bounded and non-negative Borel measurable function with bounded support.
	Then $\widehat{f} \in L^{2}(\xd)$ and 
	\[
	\int_{\xd} \Big(\widehat{f} - \int_{\xd} \widehat{f} \, d\md \Big)^2 \, d\md = 
	\frac{1}{\zeta(d)} \sum_{p,q \geq 1} \Big( \int_{\bR^d} f(pz) f(qz) \, dz + \int_{\bR^d} f(pz) f(-qz) \, dz \Big),
	\]
	where $\zeta$ denotes the Riemann zeta-function.
\end{theorem}

\vspace{0.2cm}

For a future reference, we also note that a straightforward application of  the Cauchy-Schwarz inequality to the expression in 
Theorem \ref{thm_rogers} yields the following corollary:

\vspace{0.2cm}

\begin{corollary}
	\label{cor_RogersL2}
	If $d \geq 3$ and $f : \bR^d \ra \bR$ is a bounded and non-negative Borel measurable function with bounded support, then
	\[
	\big\|\widehat{f}\big\|^2_{L^2(\xd)} \leq \|f\|_1^2 + 2 \frac{\zeta({d}/{2})^2}{\zeta(d)} \|f\|_2^2.
	\]
\end{corollary}

\vspace{0.2cm}

Now using Theorem \ref{thm_rogers}, we compute the variance:

\vspace{0.2cm}

\begin{corollary}\label{cor:l2}
\[
\sigma^2 := \lim_{T\to\infty} \frac{\int_{\xd} \left(\widehat{\chi}_{\Omega_T} - \int_{\xd} \widehat{\chi}_{\Omega_T}\right)^2 \, d\md}{\Vol(\Omega_T)}=
\frac{1}{\zeta(d)} \left(\sum_{p,q=1}^\infty  \frac{\Leb\left(p^d I \cap q^d I\right)}{p^d q^d \Leb(I)}\right) \Big( 1 +   \frac{\kd(B \cap - B)}{\kd(B)} \Big).
\]
\end{corollary}

\vspace{0.2cm}

\begin{proof}
By Theorem \ref{thm_rogers}, 
$$
\int_{\xd} \Big(\widehat{\chi}_{\Omega_T} - \int_{\xd} \widehat{\chi}_{\Omega_T}\Big)^2 \, d\md=
\frac{1}{\zeta(d)} \sum_{p,q \geq 1} \Big( \Vol(p^{-1}\Omega_T \cap q^{-1}\Omega_T)  + 
\Vol(p^{-1}\Omega_T \cap -q^{-1}\Omega_T)\Big).
$$
If we split this sum into sums over $\{p = q\}$ and $\{p \neq q\}$ and use the symmetry of $p$ and $q$ and the formula $\Vol(q^{-1}\Omega_T) = q^{-d}\Vol(\Omega_T)$ for every $q \geq 1$, we see that 
this sum can be written as
\[
\Vol(\Omega_T)+\Vol(\Omega_T \cap -\Omega_T)+
\frac{2}{\zeta(d)} \sum_{q=1}^\infty \frac{1}{q^d} \sum_{p=1}^{q-1}  \Big( \Vol\big(\Omega_T \cap (q/p)\Omega_T\big)  + 
	\Vol\big(\Omega_T \cap - (q/p) \Omega_T\big)\Big).
\]
We observe that for $c,T>0$, $I\subset (0,\infty)$, and $B \subset \bS_{\bd}$,
\[
\pm c \Omega_T(I,B) = \Omega_{cT}(c^d I, \pm B),
\]
and 
for $T_1,T_2>0$, $I_1, I_2 \subset (0,\infty)$, and $B_1, B_2\subset \bS_{\bd}$,
\[
\Omega_{T_1}(I_1,B_1) \cap \Omega_{T_2}(I_2,B_2) = \Omega_{\min(T_1,T_2)}(I_1 \cap I_2,B_1 \cap B_2).
\]
Hence, we deduce from Lemma \ref{lemma_volume} that for every $c\ge 1$,
\[
\kappa_{\pm}(c) := \lim_{T\to\infty} \frac{\Vol(\Omega_T \cap \pm c \Omega_T)}{\Vol(\Omega_T)}=
\lim_{T\to\infty} \frac{\Vol\big(\Omega_{T}(I \cap c^d I,B \cap \pm B)\big)}{\Vol(\Omega_T)}=
\frac{\Leb(I \cap c^d I)}{\Leb(I)} \, \frac{\kd(B \cap \pm B)}{\kd(B)}.
\]
Then since we are assuming that $d\ge 3$, we can apply the Dominated Convergence Theorem to conclude that the limit $\sigma^2$ exists and 
\begin{align*}
\sigma^2 &= 1 +\frac{\kd(B\cap -B)}{\kd(B)}+  \frac{2}{\zeta(d)} \sum_{q=1}^\infty \frac{1}{q^{d}} \sum_{p=1}^{q-1}  \big(\kappa_{+}(q/p) + \kappa_{-}(q/p)\big)\\
&= \left(1 +  \frac{2}{\zeta(d)} \sum_{q=1}^\infty \frac{1}{q^{d}} \sum_{p=1}^{q-1} \frac{\Leb\left(I \cap (q/p)^d I\right)}{\Leb(I)} \right) \Big( 1 +   \frac{\kd(B \cap - B)}{\kd(B)} \Big).
\end{align*}
This implies the stated formula.
\end{proof}

\subsection{Tessellations of the sets $\Omega_T(I,B)$}

In this subsection, we construct, for all large enough $T$, a {functional} tiling of the indicator function $\chi_{\Omega_T}$
using the coordinate system introduced in the previous section.
This tiling will be the basis for our smooth approximation scheme later.  Before we can state our main observation (Corollary \ref{cor_exacttiling}) of this subsection, we need some preliminaries. For a positive integer $N$, we define
\[
\cS(N) := \left\{ (u_1,\ldots,u_{k-1}) \in \bR^{k-1} \, : \,\, u_1,\ldots,u_{k-1} < 0 \qen {\sum}_{j=1}^{k-1} u_j > -N \right\},
\]
and set
\begin{equation}
\label{def_S1S2}
\cS_1 := \cS(1) \qand \cS_2 := [-1,0)^{k-1} \setminus \cS(1).
\end{equation}
We note that this definition of $\cS_1$ coincides with the one given in \eqref{eq:s1} above. Furthermore, we define
\[
P_{N,i} := \big\{ n \in [0,N]^{k-1} \cap \bZ^{k-1} \, :\,\, \cS_i - n \subset S(N) \big\}, \quad \textrm{for $i = 1,2$}.
\]
\vspace{0.2cm}

\begin{lemma}\label{l:p}
For every positive integer $N$, 
\[
\cS(N) = \Big( \bigsqcup_{n \in P_{N,1}} (\cS_1-n) \Big) \bigsqcup  \Big( \bigsqcup_{n \in P_{N,2}} (\cS_2-n) \Big).
\]
In particular, 
\[
\max_{n \in P_{N,i}} |n| \ll N \qand |P_{N,i}| \ll \Vol_{k-1}(S(N)) \ll N^{k-1}.
\] 
\end{lemma}

\vspace{0.2cm}

\begin{proof}
Fix $u \in \cS(N)$, and note that since $-N \leq u_j \leq 0$ for all $j$, there are \emph{unique} integers $0 \leq n_j \leq N$ such that 
\[
w := u + n \in [-1,0)^{k-1}, \quad \textrm{where $n = (n_1,\ldots,n_{k-1})$}, 
\]
and thus either $w \in \cS_1$ or $w \in \cS_2$, whence $u \in \cS_{i} - n$ for either $i = 1,2$. Clearly these are disjoint events, so 
in particular,
\[
S(N) = \Big( \bigsqcup_{n \in P_{N,1}} (S_1-n) \Big) \bigsqcup  \Big( \bigsqcup_{n \in P_{N,2}} (S_2-n) \Big),
\]
which finishes the proof.
\end{proof}

We observe that 
in view of  \eqref{eq:delta}
the sets $\Delta_T$ are related to suitable dilations of the sets $\cS(N)$. 
Indeed, for $T$ and $s$ with $s < d \log T$, we let 
$$
\tau_T(s): = \Diag\big(\tau_{T,1}(s),\ldots,\tau_{T,k-1}(s)\big)
$$
denote the 
diagonal $(k-1) \times (k-1)$-matrix with the {positive} diagonal entries
\begin{equation}
\label{deftau}
\tau_{T,j}(s): = \frac{d\log T-s}{d_j \lfloor \log T \rfloor}, \quad \textrm{for $j=1,\ldots,k-1$},
\end{equation}
then 
\[
\Delta_T=\big\{(u,s,\xi)\, :\,\, s \in I,\quad \xi \in B,\, \quad u \in \tau_T(s)\cS(\lfloor \log T \rfloor) + v_T\big\}.
\]
Therefore, applying the Lemma \ref{l:p} to $\cS(\lfloor \log T \rfloor)$, we get the following ``functional tiling'' for the characteristic function $\chi_{\Delta_T}$.

\vspace{0.2cm}

\begin{lemma}\label{l:tile}
For all $(u,s,\xi) \in \bR^{k-1} \times \bR \times \bS_{\bd}$ with $s < d\log T$, 
\begin{align*}
\chi_{\Delta_T}(u,s,\xi) 
=& 
\sum_{n \in P_{\lfloor \log T \rfloor,1}} \chi_{\cS_1}\left(\tau_T(s)^{-1}(u + \tau_T(s)n - v_T)\right) \, \chi_{\log I}(s) \, \chi_B(\xi)  \\
&\;\;\;\;+ 
\sum_{n \in P_{\lfloor \log T \rfloor,2}} \chi_{\cS_2}\left(\tau_T(s)^{-1}(u + \tau_T(s)n - v_T)\right) \, \chi_{\log I}(s) \, \chi_B(\xi).
\end{align*}
In particular, for all $T> e^{\sup(I)/d}$, this identity holds everywhere. 
\end{lemma}

\subsection{Construction of a functional tiling}

Now we construct our functional tiling, namely, the objects satisfying conditions (I.a)--(I.c) and (II.a)--(II.c) with $V_T:=\Vol(\Omega_T)$.

\subsubsection{Construction of the sets $\widetilde Q_{T,i}$, $Q_{T,i}(y)$ and maps $\beta_{T,i}$, $\widetilde{\beta}_{T,i}$
(assumptions (II.a)--(II.b))}
\label{subsec:constructbeta}
Let us now rewrite the assertion of Lemma \ref{l:tile},
so that it fits the decomposition \eqref{FT0}. 
We note that
\begin{equation}
\label{deftauinfty}
\tau_{T}(s)=\tau_\infty+O\big(1/(\log T)\big)\quad\hbox{as $T\to\infty$}
\end{equation}
uniformly on $s$ in compact sets, where
$$
\tau_\infty := \Diag(d/d_1,\ldots,d/d_{k-1}).
$$
We define
\[
\beta_{T} : \bR^{k-1} \times \bR \ra \bR^{k-1} \qand \widetilde{\beta}_{T} : \bR^{k-1} \ra \bR^{k-1}
\]
by
\begin{equation}
\label{eq:betta}
\beta_{T}(u,s) := \tau_T(s)u - v_T \quad\hbox{and}\quad
\widetilde{\beta}_{T}(u) := \tau_\infty u - v_T
\end{equation}
for $u \in \bR^{k-1}$ and $s \in \bR$.
Let
\begin{equation}\label{defQTSitilde0}
\widetilde{Q}_{T,i} := P_{\lfloor \log T \rfloor,i} \subset \bR^{k-1}, \quad\quad \textrm{for $i = 1,2$}.
\end{equation}
From Lemma \ref{lemma_volume} and Lemma \ref{l:p}, we see that
$|\widetilde{Q}_{T,i}|  \ll \Vol(\Omega_T).$
The condition \eqref{QTgrowth} in (II.a) can be also checked easily.
The following lemma verifies condition (II.b).
We recall that $\|\cdot\|$ denotes the $\ell^\infty$-norm on $\bR^{k-1}$.

\vspace{0.2cm}

\begin{lemma}
\label{lemma_beta}
Let $J \subset \bR$ be a bounded interval.
\begin{enumerate}
\item[(i)] There exist $c_1,c_2>0$ such that for all $T\ge T_0(J)$, $s_1,s_2 \in J$,
and $u,v \in \widetilde{Q}_{T,i}$,
$$
\big\|\beta_{T}(u,s_1) - \beta_{T}(v,s_2)\big\| \geq c_1 \|u-v\|-c_2.
$$
\item[(ii)] There exists $c_3>0$ such that for all $T\ge T_0(J)$, $s \in J$, and $u \in \widetilde{Q}_{T,i}$,
\[
\big\|\beta_T(u,s) - \widetilde{\beta}_T(u)\big\| \leq c_3.
\]
\end{enumerate}
\end{lemma}

\vspace{0.2cm}

\begin{proof}
Since $\|u\| \ll \log T$ for all $u \in \widetilde{Q}_{T,i}$,
this lemma follows immediately from \eqref{deftauinfty} and the definitions of the maps
$\beta_T$ and $\tilde \beta_T$.
\end{proof}

\begin{remark}
While in Section \ref{sec:outline} we have allowed $\beta_{T}$ and $\widetilde{\beta}_{T}$ to also depend on $i$, it is not necessary at this point. However, to properly work with these functions in our
setting, we also need to define the finite measure spaces $(Y_{T,i},\kappa_{T,i})$, for $i= 1,2$. This will be done in the next section. 
\end{remark}

Let us now rewrite the decomposition in Lemma \ref{l:tile} using the standard coordinates. We set
\begin{equation}
\label{defQTSitilde}
Q_{T,i}(s) := \beta_T\big(\widetilde{Q}_{T,i},s\big) 
\end{equation}
and
$$
\widetilde{h}_{T,i}(u,s,\xi): = \chi_{\cS_i}\left(\tau_T(s)^{-1}u\right) \, \chi_{\log I}(s) \, \chi_B(\xi),
$$
for $i =1,2$, and note that the assertion in the lemma above can be written as
\begin{align}
\chi_{\Delta_T}(u,s,\xi) 
=&
\sum_{w \in \tilde Q_{T,1}} \widetilde{h}_{T,1}(u+\beta_T(w,s),s,\xi)
+ 
\sum_{w \in \tilde Q_{T,2}} \widetilde{h}_{T,2}(u+\beta_T(w,s),s,\xi) \label{DeltaTh}\\
=& 
\sum_{v \in Q_{T,1}(s)} \widetilde{h}_{T,1}(u+v,s,\xi)  + \sum_{v \in Q_{T,2}(s)} \widetilde{h}_{T,2}(u+v,s,\xi) \nonumber 
\end{align}
for all large enough $T$. Let us now set
\[
h_{T,i}: = \widetilde{h}_{T,i} \circ \pid, \quad \textrm{ for $i = 1,2$}.
\]
Since $\chi_{\Omega_T} = \chi_{\Delta_T} \circ \pid$, the equivariance \eqref{equiad} 
of $\pid$ yields the following corollary of Lemma \ref{l:tile}:

\vspace{0.2cm}

\begin{corollary}
\label{cor_exacttiling}
For all large enough $T$, 
\[
\chi_{\Omega_T}(z) = \sum_{v \in Q_{T,1}(\sd(z))} h_{T,1}(\aad(v)z) + \sum_{v \in Q_{T,2}(\sd(z))} h_{T,2}(\aad(v)z),\quad \hbox{for $z \in \bR_{*}^{\bd}$.}
\]

\end{corollary}

\vspace{0.2cm}

We stress that the summation range in the above formula depend on the point $z$, albeit in a weak way via $\sd(z)$.
In the next subsection, we will get rid of this $z$-dependence upon introducing an additional average. The price we have to pay for this
is that the functions $h_{T,i}$ will be replaced with more complicated  functions $f_{T,i}$, which depend on the an extra variable, 
coming from the average. 

\subsubsection{Construction of the spaces $(Y_{T,i},\kappa_{T,i})$ and functions $f_{T,i}$ (assumptions (I.a)--(I.b))}


If $\cT \subset \bR^{k-1}$
is a subset and $r \geq 0$, we denote by $\cT_r$ the $r$-thickening of $\cT$ with respect to this norm. Similarly, for a subset $B$ of $\bS_{\bd}$,
we denote by $B_r$ the $r$-thickening of $B$ 
with respect to the rotation-invariant metric on $\bS_{\bd}$. \\

Since $|v| \ll \log T$
for every $v \in \widetilde{Q}_{T,i}$, it follows from \eqref{deftauinfty} 
that for any bounded interval $J \subset \bR$, there exist $c(J)>0$ such that
for all $s,t \in J$, $T \ge T_0(J)$, and $v \in \widetilde{Q}_{T,i}$,
\[
 \big\|\tau_T(s)^{-1}(\beta_T(v,s) - \beta_T(v,t))\big\|=
 \big\|\tau_T(s)^{-1}(\tau_T(s)v - \tau_T(t)v)\big\| \leq c(J)\, |s-t|.
\]
Hence, we deduce that for all $s,t \in J$ satisfying $|s-t| \leq r$,
$T \ge T_0(J)$, $u\in \mathbb{R}^{k-1}$, and $v \in \widetilde{Q}_{T,i}$,
\begin{equation}
\label{why}
\chi_{S_i}\big(\tau_T(s)^{-1}(u + \beta_T(v,s))\big)  \leq \chi_{(S_i)_{c(J) r}}\big(\tau_T(s)^{-1}(u + \beta_T(v,t))\big).
\end{equation}
Let us now introduce a parameter $\eps\in (0,1)$ and a non-negative real smooth function $\rho_\eps$ on $\bR$ with
\begin{equation}\label{eq:rrrr}
\supp (\rho_\eps) \subset [-\eps/2,\eps/2] \qand \int_{\bR} \rho_\eps(t) \, dt = 1.
\end{equation}
For future reference, we also note that $\rho_\eps$ can be chosen, so that 
\begin{equation}
\label{eq:norm1}
\|\rho_{\eps_T}\|_{C^q}\ll \eps^{-1-q}.
\end{equation}
By the standard properties of convolutions,
$$
\chi_{\log I} \leq \rho_\eps * \chi_{(\log I)_\eps} \leq \chi_{(\log I)_{2\eps}}.
$$
Then, using \eqref{why}, we deduce that for every $u\in\mathbb{R}^{k-1}$ and $v \in \widetilde{Q}_{T,i}$,
\begin{eqnarray*}
\widetilde{h}_{T,i}\big(u + \beta_T(v,s),s,\xi\big) 
&=& 
\chi_{S_i}\big(\tau_T(s)^{-1}(u + \beta_T(v,s))\big) \chi_{\log I}(s) \, \chi_B(\xi) \\
&\leq &
\int_{(\log I)_\eps} \chi_{S_i}\big(\tau_T(s)^{-1}(u + \beta_T(v,s))\big) \rho_\eps(s-t) \, \chi_B(\xi) \, dt \\
&= &
\int_{(\log I)_\eps} \chi_{(S_i)_{c \eps}}\big(\tau_T(s)^{-1}(u + \beta_T(v,t))\big) \, \rho_\eps(s-t) \, \chi_B(\xi) \, dt,
\end{eqnarray*}
where $c=c(J)>0$ for a fixed bounded interval $J$ which contains $(\log I)_\eps$ for all $0 < \eps < 1$. 
Let $\psi_{i,\eps}$ be a smooth function on $\bR^{k-1}$ such that
\begin{equation}\label{eq:chi_0}
\chi_{(\cS_i)_{c\eps}} \leq \psi_{i,\eps} \leq \chi_{(S_i)_{2c\eps}}, \quad \textrm{for $i = 1,2$},
\end{equation}
and let $\vartheta_\eps$ be a smooth function on $\bS_{\bd}$ such that
\begin{equation}\label{eq:chi_1}
\chi_{B} \leq \vartheta_{\eps} \leq \chi_{B_{\eps}}.
\end{equation}
For future reference, we note that these functions can be constructed, so that
\begin{equation}\label{eq:norm2}
\|\psi_{i,\eps}\|_{C^q}\ll \eps^{-1-q}\quad\hbox{and}\quad \|\vartheta_{\eps}\|_{C^q}\ll \eps^{-\theta_q}\quad \hbox{for some $\theta_q>0$.}
\end{equation}
From the above estimate, we deduce that for every $u\in\mathbb{R}^{k-1}$ and $v \in \widetilde{Q}_{T,i}$,
\begin{equation}\label{eq:hhh}
\widetilde{h}_{T,i}\big(u + \beta_T(v,s),s,\xi\big) 
\leq 
\int_{(\log I)_\eps} \psi_{i,\eps}\big(\tau_T(s)^{-1}(u + \beta_T(v,t))\big) \, \rho_\eps(s-t) \, \vartheta_\eps(\xi) \, dt.
\end{equation}
By the same argument as in \eqref{why}, we also have
for all $s,t \in J$ satisfying $|s-t| \leq \eps$,
$T \ge T_0(J)$, $u\in \mathbb{R}^{k-1}$, and $v \in \widetilde{Q}_{T,i}$,
\[
\chi_{(S_i)_{2 c \eps}}\big(\tau_T(s)^{-1}(u + \beta_T(v,t)\big)  \leq \chi_{(S_i)_{ 3 c \eps}}\big(\tau_T(s)^{-1}(u + \beta_T(v,s))\big).
\]
Then it follows from \eqref{eq:chi_0} and \eqref{eq:chi_1} that 
\begin{eqnarray}\label{3eps}
\widetilde{h}_{T,i}\big(u + \beta_T(v,s),s,\xi\big) 
&\leq & 
\int_{(\log I)_\eps} \psi_{i,\eps}\big(\tau_T(s)^{-1}(u + \beta_T(v,t))\big) \, \rho_\eps(s-t) \, \vartheta_\eps(\xi) \, dt,
\end{eqnarray}
and 
\begin{align}
&\int_{(\log I)_\eps} \psi_{i,T}\big(\tau_T(s)^{-1}(u + \beta_T(v,t))\big) \, \rho_\eps(s-t) \, \vartheta_\eps(\xi) \, dt \nonumber \\
\leq\, &
\int_{(\log I)_{\eps}} \chi_{(\cS_i)_{2c \eps}}\big(\tau_T(s)^{-1}(u + \beta_T(v,t))\big) \, \rho_\eps(s-t) \, \vartheta_\eps(\xi) \, dt \nonumber \\
\leq\, &
\int_{(\log I)_{\eps}} \chi_{(\cS_i)_{3c \eps}}\big(\tau_T(s)^{-1}(u + \beta_T(v,s))\big) \, \rho_\eps(s-t) \, \vartheta_\eps(\xi) \, dt \nonumber \\
\leq\, &
\chi_{(\cS_i)_{3c \eps}}\big(\tau_T(s)^{-1}(u + \beta_T(v,s))\big) \, \chi_{(\log I)_{2\eps}}(s) \, \chi_{B_\eps}(\xi). \label{3eps_1}
\end{align}

\medskip

We introduce a parameter $\eps_T\in (0,1)$, to be specified later, and define 
\begin{equation}
\label{eq:Y}
Y_{T} := (\log I)_{\eps_T} \qand \kappa_T := \Leb \mid_{Y_T}.
\end{equation}
For $y \in Y_{T}$, we set
\begin{equation}
\label{deffTi}
\widetilde{f}_{T,i}\big((u,s,\xi),y\big) := \psi_{i,\eps_T}(\tau_T(s)^{-1}u) \, \rho_{\eps_T}(s-y) \, \vartheta_{\eps_T}(\xi)
\end{equation}
and consider
\begin{align*}
\widetilde{F}_T(u,s,\xi)
:=& 
\int_{Y_T} \Big( \sum_{w \in Q_{T,1}(y)} \widetilde{f}_{T,1}\big((u+w,s,\xi),y\big) \Big) \, d\kappa_T(y) \\
&+
\int_{Y_T} \Big( \sum_{w \in Q_{T,2}(y)} \widetilde{f}_{T,2}\big((u+w,s,\xi),y\big) \Big) \, d\kappa_T(y). 
\end{align*}
It follows from \eqref{eq:hhh} that for every $u\in\mathbb{R}^{k-1}$ and $v \in \widetilde{Q}_{T,i}$,
\[
\widetilde{h}_{T,i}\big(u + \beta_T(v,s),s,\xi\big) \leq \int_{Y_T} \widetilde{f}_{T,i}\big((u + \beta_T(v,y),s,\xi),y\big) \, d\kappa_{T}(y).
\]
Hence, by \eqref{DeltaTh} and  \eqref{3eps},
\begin{align*}
\chi_{\Delta_T}(u,s,\xi) 
\leq & 
\int_{Y_T} \Big( \sum_{v \in \widetilde{Q}_{T,1}} \widetilde{f}_{T,1}\big((u+\beta_T(v,y),s,\xi),y\big) \Big) \, d\kappa_T(y) \\
&\quad +
\int_{Y_T} \Big( \sum_{v \in \widetilde{Q}_{T,2}} \widetilde{f}_{T,2}\big((u+\beta_T(v,y),s,\xi),y\big) \Big) \, d\kappa_T(y)\\
=&\,\widetilde{F}_T(u,s,\xi).
\end{align*}
Let 
\begin{align}\label{eq:chiplus}
\chi_{\Delta_T}^{+}(u,s,\xi)
:=&
\sum_{v \in \widetilde{Q}_{T,1}} \chi_{(\cS_1)_{3 c \eps_T}}\big(\tau_T(s)^{-1}(u + \beta_T(v,s))\big) \, \chi_{(\log I)_{2\eps_T}}(s) \, \chi_{B_{\eps_T}}(\xi)  \\
&\;\;\;\;+
\sum_{v \in \widetilde{Q}_{T,2}} \chi_{(\cS_2)_{3 c \eps_T}}\big(\tau_T(s)^{-1}(u + \beta_T(v,s))\big) \, \chi_{(\log I)_{2\eps_T}}(s) \, \chi_{B_{\eps_T}}(\xi).\nonumber
\end{align}
Then it follows from \eqref{3eps_1} that 
$$
\widetilde{F}_T(u,s,\xi) \le \chi_{\Delta_T}^{+}(u,s,\xi).
$$
We conclude that 
\begin{equation}
\label{eq:fff}
\chi_{\Delta_T}\le \widetilde{F}_T\le \chi_{\Delta_T}^{+}.
\end{equation}
The estimate indicates that $\widetilde{F}_T$ provides 
an approximation for the characteristic function $\chi_{\Delta_T}$.
Let us now define $f_{T,i} : \bR^d \times \bR \ra [0,\infty)$ by
\begin{equation}\label{eq:f_tilde}
f_{T,i}(z,y) = \widetilde{f}_{T,i}(\pid(z),y) \quad \textrm{for $z \in \bR_*^{\bd}$}\;\;\hbox{and}\;\; y \in Y_{T},
\end{equation}
and $f_{T,i}(z,y) := 0$ for all $z \in \bR^d \setminus \bR_*^{\bd}$. Then $f_{T,i}$ is smooth in the $z$-coordinate. We also set 
\begin{equation}
\label{defFT2}
F_T := \widetilde{F}_{T} \circ \pid.
\end{equation} 
From \eqref{equiad} we see that the function $F_T$ can be written as
\begin{equation}
\label{specialformFT}
F_T(z) = \int_{Y_T} \Big( \sum_{v \in Q_{T,1}(y)} f_{T,1}\big(\aad(v)z,y\big) \Big) \, d\kappa_T(y) \\
+
\int_{Y_T} \Big( \sum_{v \in Q_{T,2}(y)} f_{T,2}\big(\aad(v)z,y\big) \Big) \, d\kappa_T(y),
\end{equation}
which is exactly the form of functional tiling analyzed in Section \ref{sec:general}. \\

The following lemma demonstrates that the function $F_T$ proves a good approximation 
for the characteristic function $\chi_{\Omega_T}=\chi_{\Delta_T}\circ \pi$. 

\vspace{0.2cm}

\begin{lemma}
	\label{lemma_eta}
	For $\eps_T = \Vol(\Omega_T)^{-\eta}$ with $\eta>p/2$, then
	\[
	\big\|\chi_{\Omega_T} - F_T\big\|_{L^p} = o\big(\Vol(\Omega_T)^{1/2}\big)\quad\hbox{as $T\to\infty$}.
	\] 
\end{lemma}

\vspace{0.2cm}

\begin{proof}
We shall use the integral formula from Lemma \ref{lemma_volume}.
From \eqref{eq:fff},
\begin{align*}
\big\|\chi_{\Delta_T} - \tilde F_T\big\|_{L^p}
&\le \big\|\chi^+_{\Delta_T} -\chi_{\Delta_T} \big\|_{L^p}\\
&\ll
\left(\int_{\bS_{\bd}} \int_{\bR^{k-1}} \int_{\bR} \big| \chi^+_{\Delta_T}(u,s,\xi) - \chi_{\Delta_T}(u,s,\xi) \big|^p \, e^{s} \, ds \, du \, d\kd(\xi)\right)^{1/p}.
\end{align*}
We recall that $\chi^+_{\Delta_T}$ and $\chi_{\Delta_T}$
are given by \eqref{eq:chiplus} and \eqref{DeltaTh} respectively.
By successive use of the triangle-inequality, this expression
is less than $A_1 + A_2$, where 
$$
A_i:=\sum_{v\in \widetilde{Q}_{T,i}} (A_{i,1}(v)+A_{i,2}(v)+A_{i,3}(v))
$$ 
with
\begin{align*}
A_{i,1}(v) &:=  \left(\int_{\bS_{\bd}} \int_{\bR^{k-1}} \int_{\bR} \chi_{(S_i)_{3 c \eps_T} \setminus S_i}\big(\tau_T(s)^{-1}(u + \beta_T(v,s))\big) \chi_{J}(s) \, \chi_{B_{\eps_T}}(\xi) \, e^{s} \, ds \, du \, d\kd(\xi)\right)^{1/p},\\
A_{i,2}(v) &:=  \left(\int_{\bS_{\bd}} \int_{\bR^{k-1}} \int_{\bR} \chi_{(S_i)_{3 c \eps_T}}\big(\tau_T(s)^{-1}(u + \beta_T(v,s))\big) 
\chi_{(\log I)_{2\eps_T} \setminus (\log I)_{\eps_T}}(s) \, \chi_{B_{\eps_T}}(\xi) \, e^{s} \, ds \, du \, d\kd(\xi)\right)^{1/p}, \\
A_{i,3}(v) &:=  \left(\int_{\bS_{\bd}} \int_{\bR^{k-1}} \int_{\bR} \chi_{(S_i)_{3 c \eps_T}}\big(\tau_T(s)^{-1}(u + \beta_T(v,s))\big) \chi_{J}(s) \, 
\chi_{B_{\eps_T} \setminus B}(\xi) \, e^{s} \, ds \, du \, d\kd(\xi)\right)^{1/p}.
\end{align*}
Since
$$
\Leb_{k-1}\big(\tau_T(s)((S_i)_{3 c \eps} \setminus S_i)\big) \ll \eps
$$
uniformly over $s\in J$  and sufficiently large $T$,
we conclude that $A_{i,1}(v)\ll \eps_T^{1/p}$ uniformly over $v$.
Also since 
$$
\Leb_{k-1}\big(\tau_T(s)(S_i)_{3 c \eps}\big) \ll 1
$$
uniformly over $s\in J$  and sufficiently large $T$, and 
$$
\Leb_1\big((\log I)_{2\eps} \setminus (\log I)_{\eps}\big) \ll \eps\quad\hbox{and}\quad
\kd(B_{\eps} \setminus B) \ll \eps,
$$
we deduce that $A_{i,2}(v)+A_{i,3}(v)\ll \eps_T^{1/p}$ uniformly on $v$. Therefore,
\[
\big\|\chi_{\Delta_T} - \tilde F_T\big\|_{L^p} 
\ll \big(|\widetilde{Q}_{T,1}|+|\widetilde{Q}_{T,2}|\big) \, \eps_T^{1/p} \ll \Vol(\Omega_T) \, \eps_T^{1/p}.
\]
Hence, when $\eps_T = \Vol(\Omega_T)^{-\eta}$ with $\eta>p/2$, 
we have $\|\chi_{\Omega_T} - F_T\|_{L^p} = o\big(\Vol(\Omega_T)^{1/2}\big).$
\end{proof}

\subsubsection{Construction of the maps $h_{T,i}$ (assumption (II.c))}\label{sec:h}

Let us now turn to the construction of the maps $h_{T,i}$
satisfying the condition (II.c).  
We recall that $h_{T,i}$ should be non-negative 
Borel functions on $\bR^d \times Y_{T,i}$ satisfying
\[
f_{T,i}\big(\aad(\beta_T(v,y))z,y\big) \leq h_{T,i}\big(\aad(\widetilde{\beta}_T(v))z,y\big)
\]
for all $v \in \widetilde{Q}_{T,i}$,  $z \in \bR^d$, and $y \in Y_{T,i}$.
Moreover, we arrange that the supports the functions $x \mapsto h_{T,i}(x,y)$ lie in a fixed compact set, independent of $y \in Y_{T}$, and 
\[
\sup_{z,T} \int_{Y_{T}} h_{T,i}(z,y) \, d\kappa_{T}(y) < \infty.
\]
We shall use the coordinate system \eqref{defphi}. 
Then in view of \eqref{eq:f_tilde}, 
it is sufficient to construct non-negative Borel functions $\widetilde{g}_{T,i}$ on $(\bR^{k-1} \times \bR \times \bS_{\bd}) \times Y_T$ such that 
\begin{equation}\label{eq:est0}
\widetilde{f}_{T,i}\big((u + \beta_T(v,y),s,\xi),y\big) \leq \widetilde{g}_{T,i}\big((u + \widetilde{\beta}_T(v),s,\xi),y\big),
\end{equation}
for all $v \in \widetilde{Q}_{T,i}$ and $(u,s,\xi) \in \bR^{k-1} \times \bR \times \bS_{\bd}$ and $y \in Y_{T}$, 
whose supports lie in a set $\cK \times Y_{T,i}$, with a fixed compact $\cK \subset \bR^{k-1} \times \bR \times \bS_{\bd}$, and such that
\[
\sup_{(u,s,\xi), T} \int_{Y_{T}} \widetilde{g}_{T,i}\big((u,s,\xi),y\big) \, d\kappa_{T}(y) <\infty.
\]
Indeed, if such maps have been constructed, we can simply set $h_{T,i} = \widetilde{g}_{T,i} \circ \pid$.
We recall from \eqref{deffTi} that
\[
\widetilde{f}_{T,i}\big((u,s,\xi),y\big) = \psi_{i,\eps_T}(\tau_T(s)^{-1}u) \, \rho_{\eps_T}(s-y) \, \vartheta_{\eps_T}(\xi),
\]
where $\psi_{i,\eps_T}$ satisfies
\[
\chi_{(\cS_i)_{c\eps_T}} \leq \psi_{i,\eps_T} \leq \chi_{(S_i)_{2c\eps_T}}.
\]
By Lemma \ref{lemma_beta}(ii), there is a compact set $\cD_0 \subset \bR^{k-1}$ such that \[
\beta_{T}(v,y) - \widetilde{\beta}_T(v) \in \cD_0, \quad \textrm{for all $v \in \widetilde{Q}_{T,i}$, $y \in Y_{T}$, and $T\ge T_0(J)$}.
\]
Furthermore, by the construction of the map $\tau_T$ in \eqref{deftau}, there exists a compact set $\cD \subset \bR^{k-1}$
such that 
\[
\tau_T(s) \Big( (S_i)_{2c \eps_T} - \cD_0\Big) \subset \cD, \quad \textrm{for all $s \in J$ and sufficently large $T$}.
\]
Hence, for all $s \in J$, $u\in\bR^{k-1}$, $v \in \widetilde{Q}_{T,i}$, $y\in Y_T$,  and sufficiently large $T$,
\begin{eqnarray*}
\psi_{i,\eps_T}\left(\tau_T(s)^{-1}(u + \beta_T(v,y))\right)
& \leq &
\chi_{(S_i)_{2c \eps_T}}\left(\tau_T(s)^{-1}(u + \widetilde{\beta}_T(v) + \beta_{T}(v,y)-\widetilde{\beta}_T(v))\right) \\ 
&\leq &
\chi_{(S_i)_{2c \eps_T} - \cD_0}\left(\tau_T(s)^{-1}(u + \widetilde{\beta}_T(v))\right) \\
&\leq &
\chi_{\cD}(u + \widetilde{\beta}_T(v)).
\end{eqnarray*}
Let us now define
\begin{equation*}
\widetilde{g}_{T,i}\big((u,s,\xi),y\big) := \chi_{\cD}(u) \, \rho_{\eps_T}(s-y) \, \vartheta_{\eps_T}(\xi).
\end{equation*}
Then the estimate \eqref{eq:est0} clearly holds. Furthermore, 
\[
\int_{Y_{T,i}} \widetilde{g}_{T,i}((u,s,\xi),y) \, d\kappa_T(y) \le \int_{J}  \chi_{\cD}(u) \rho_{\eps_T}(s-y) \, \vartheta_{\eps_T}(\xi) \, dy
\leq \chi_{\cD}(u) \chi_{J_{\eps_T}}(s) \, \chi_{B_{\eps_T}}(\xi),
\]
which is clearly compactly supported and bounded, uniformly in $T$.

\subsection{Estimation of the function norms}
\label{subsec:MTq}

In order to apply our general result from the previous section (Theorem \ref{thm_criterion}),
We have to estimate the norms of the functions $f_{T,i}$, specifically, the quantities
$M_{T}$ and $M_{T,q}$ defined in \eqref{defM0_0}--\eqref{defM0}.

\vspace{0.2cm}

\begin{lemma}\label{l:norms}
For the functions $f_{T,i}$ defined in \eqref{eq:f_tilde},
$$
{M}_{T} \ll \eps_T^{-1}\quad\hbox{and}\quad {M}_{T,q} \ll \eps_T^{-r_q}
$$ 
with $r_q>0$.
\end{lemma}

\vspace{0.2cm}

\begin{proof}
We use that $f_{T,i}(\cdot, y) = \widetilde{f}_{T,i}(\cdot,y) \circ \pid$, and 
the maps $\widetilde{f}_{T,i}(\cdot,y)$ are supported in a fixed compact
subset of $\bR^{k-1} \times \bR \times \bS_{\bd}$, which is independent of $y \in Y_{T,i}$. Then the restrictions to this compact set of all
partial derivatives of the map $\pid$ are uniformly bounded. Therefore,
it is sufficient to estimate 
\[
\widetilde{M}_{T} := \max_{i} \int_{Y_{T}} \big\|\widetilde{f}_{T,i}(\cdot,y)\big\|_{C^0}\, d\kappa_{T,i}(y)
\quad\hbox{and}\quad
\widetilde{M}_{T,q} := \max_{i} \sup_{y_i \in Y_{T}} \big\|\widetilde{f}_{T,i}(\cdot,y)\big\|_{C^q}.
\]
We recall from \eqref{deffTi} that 
\[
\widetilde{f}_{T,i}\big((u,s,\xi),y\big) = \psi_{i,\eps_T}\big(\tau_T(s)^{-1}u\big) \, \rho_{\eps_T}(s-y) \, \vartheta_{\eps_T}(\xi).
\]
According to \eqref{eq:chi_0}, \eqref{eq:norm1}, and \eqref{eq:chi_1},
$$
\|\psi_{i,\eps_T}\|_{C^0}\le 1,\quad\quad \|\rho_{\eps_T}\|_{C^0}\ll \eps_T^{-1},\quad\quad \|  \vartheta_{\eps_T}\|_{C^0}\le 1.
$$
Hence, we conclude that $\big\|\widetilde{f}_{T,i}(\cdot,y)\big\|_{C^0}\ll \eps_T^{-1}.$ This proves the first estimate.
Using additionally \eqref{eq:norm2}, we conclude that also 
$\big\|\widetilde{f}_{T,i}(\cdot,y)\big\|_{C^q}\ll \eps_T^{-r_q}$ 
for some $r_q>0$, which implies the second estimate.
\end{proof}

\subsection{Proof of Theorem \ref{main}}
\label{subsec:punchline}

Let us now summarize what we have done in this technical section. 
The aim has been to produce a smooth approximations $F_T$ 
for the indicator functions $\chi_{\Omega_T}$ to which the arguments of Section \ref{sec:general} apply. 
These approximations are given explicitly in \eqref{specialformFT}.
They are integrals of varying averages which are fibered over the finite  measure spaces 
\[
(Y_T,\kappa_T) = \left((\log I)_{\eps_T},\Leb \mid_{(\log I)_{\eps_T}}\right).
\]
These averages are constructed using finite subsets $\widetilde{Q}_{T,i}
$ and  ${Q}_{T,i}(y)$ of  $\bR^{k-1}$, defined in \eqref{defQTSitilde0} and \eqref{defQTSitilde}, and Borel maps
$\beta_T : \bR^{k-1} \times Y_T \ra \bR^{k-1}$ and  $\widetilde{\beta}_T : \bR^{k-1} \ra \bR^{k-1}$,
defined in \eqref{eq:betta}.
The approximations $F_T$ depend on a choice of a parameter $\eps_T$, which we take $\eps_T = \Vol(\Omega_T)^{-\eta}$ for some $\eta > 0$. In order
for these approximations to be useful for us, we arrange that
\begin{equation}\label{eq:norms_0}
\big\|\chi_{\Omega_T} - F_T\big\|_{L^p(X)} = o\left(\Vol(\Omega_T)^{1/2}\right)
\quad\hbox{ as $T \ra \infty$},\quad \hbox{for $p=1,2$}.
\end{equation}
According to  Lemma \ref{lemma_eta},  
one can take $\eta > 1$.  Then  \eqref{eq:norms_0} holds.
The averages are further made up by Borel functions $f_{T,i} : \bR^d \times Y_T \ra [0,\infty)$, which are defined in  \eqref{deffTi} and \eqref{eq:f_tilde}. These functions are smooth in the first variable, but unbounded as $T \ra \infty$. They are however "bounded on average", in 
the sense that there are Borel functions $h_{T,i} : \bR^d \times Y_T \ra [0,\infty)$
defined in Section \ref{sec:h}. Ultimately, this provides the framework outlined
in (I.a)--(I.c) and (II.a)--(II.c) from Section \ref{sec:general}, so that
we can apply Theorem \ref{thm_criterion} with $V_T=\hbox{Vol}(\Omega_T)$.
The conditions on the norms  $M_T$ and $M_{T,q}$ have been verified in Lemma \ref{l:norms}
with $\theta_0=\eta>1$. We recall that the limit 
$$
\sigma := \lim_{T\to\infty} V_T^{-1/2}\left\|\widehat{\chi}_{\Omega_T} - \int_{\xd} \widehat{\chi}_{\Omega_T}d\mu \right\|_{L^2(X)}.
$$
has been computed in Corollary \ref{cor:l2}. In view of \eqref{eq:norms_0},
it follows from Corollary \ref{cor_RogersL2} that 
\begin{equation}\label{eq:norm3}
\big\|\widehat{\chi}_{\Omega_T} - \widehat{F}_T\big\|_{L^p(X)}=o\big(V_T^{1/2}\big)\quad \hbox{as $T\to\infty$,}
\quad \hbox{for $p=1,2$.}
\end{equation}
Hence, we conclude that also 
$$
\lim_{T\to\infty} V_T^{-1/2}\left\|\widehat{F}_{T} - \int_{\xd} \widehat{F}_{T}d\mu \right\|_{L^2(X)}=\sigma.
$$
Now we have verified all the assumptions of Theorem \ref{thm_criterion}. \\

We conclude that the functions $V_T^{-1/2}\big(\widehat{F}_{T} - \int_{\xd} \widehat{F}_{T}d\mu\big)$ converges in distribution to the Normal Law with variance $\sigma$ when $d>4(1+\eta)$ with some $\eta>1$, namely, when $d\ge 9$. Because of \eqref{eq:norm3}, the 
functions 
$$
V_T^{-1/2}\left(\widehat{\chi}_{\Omega_T}(\Lambda) - \int_{\xd} \widehat{\chi}_{\Omega_T}d\mu\right)=
V_T^{-1/2}\Big( |\Lambda\cap \Omega_T| - \hbox{Vol}(\Omega_T)\Big)
$$
also converges in distribution to the same limit.

\end{document}